\documentclass[11pt]{article}
\usepackage[utf8]{inputenc}
\usepackage[T1]{fontenc}
\usepackage{amsmath,amssymb,amsthm,bbm,float,graphicx,geometry,lmodern,mathtools,parskip,setspace,subcaption}
\usepackage[colorlinks]{hyperref}
\usepackage{mathrsfs}
\usepackage{enumerate}
\usepackage{nicefrac}
\usepackage{todonotes}

\usepackage{tikz}
\usetikzlibrary{patterns}
\usetikzlibrary{backgrounds, positioning, snakes, shapes.geometric}
\usetikzlibrary{external, arrows} 

\usepackage{caption}
\usepackage{graphicx}
\usepackage{todonotes}
\usepackage{cleveref}
\graphicspath{{Images/}}
\allowdisplaybreaks

\newcommand{\N}{\mathbb{N}}

\newtheorem{theorem}{Theorem}[section]
\newtheorem{lemma}[theorem]{Lemma}
\newtheorem{corollary}[theorem]{Corollary}
\newtheorem{prop}[theorem]{Proposition}

\theoremstyle{definition}
\newtheorem{definition}[theorem]{Definition}
\newtheorem{remark}[theorem]{Remark}

\renewcommand{\P}{\mathbb{P}}
\newcommand{\x}{\mathbf{x}}
\newcommand{\y}{\mathbf{y}}
\newcommand{\1}{\mathbbm{1}}
\newcommand{\G}{\mathscr{G}}
\newcommand{\C}{\mathscr{C}}
\newcommand{\X}{\mathbf{X}}
\newcommand{\E}{\mathbb{E}}

\newcommand{\R}{\mathbb{R}}
\newcommand{\di}{\,\mathrm{d}}

\newcommand{\Z}{\mathbb{Z}}
\newcommand{\cG}{\mathcal{G}}

\newcommand{\e}{\mathrm{e}}
\newcommand{\deff}{\bar{\delta}_{\mathsf{eff}}}
\usepackage[style = numeric, sorting=nyt, url = false, abbreviate=false, maxbibnames=9, sortcites=true, doi = false, isbn = false, backend = biber, giveninits = true]{biblatex}
\renewbibmacro{in:}{\ifentrytype{article}{}{\printtext{\bibstring{in}\intitlepunct}}}
\bibliography{bib.bib}
\author{
	\Large\textbf{Christian M\"onch}\thanks{supported by Deutsche Forschungsgemeinschaft (DFG, German Research Foundation) -- grant no.\ 443916008/SPP2265.}
	\\Johannes Gutenberg University Mainz
	\\cmoench@uni-mainz.de
}
\title{\huge\textsc{Inhomogeneous long-range percolation in the weak decay regime}}
\date{November 24, 2026}
\usepackage{etoolbox}
\ifundef{\abstract}{}{\patchcmd{\abstract}%
	{\quotation}{\quotation\noindent\ignorespaces}{}{}}
\begin{document}
	\maketitle 
	\begin{spacing}{0.9}
		\begin{abstract} 
			We study a general class of percolation models in Euclidean space including long-range percolation, scale-free percolation, the weight-dependent random connection model and several other previously investigated models. Our focus is on the \emph{weak decay regime}, in which inter-cluster long-range connection probabilities fall off polynomially with small exponent, and for which we establish several structural properties. Chief among them are the continuity of the bond percolation function and the transience of infinite clusters. 
			
			\vspace{0.5cm}
			\noindent\normalsize{{\textbf{AMS-MSC 2020}: 05C80 (Primary), 60K35 (Secondary)}
				}
			
		\end{abstract}
	\end{spacing}
\section{Introduction and overview}\label{sec:intro}
We study \emph{(edge-)inhomogeneous percolation models} of the following type: let $\eta \subset\R^d, d\geq 1$, denote a stationary ergodic point set of unit intensity. Canonical choices of $\eta$ are a homogeneous Poisson point process or the integer lattice $\Z^d$. A more detailed discussion of the class of underlying point sets for which our results hold is given in Section~\ref{sec:Results} together with an alternative and more rigorous construction of the model. Let $\eta'$ denote an independent marking of $\eta$ by i.i.d.\ $\operatorname{Uniform}(0,1)$ random variables. We call $\x=(x,s)\in\eta'$ a \emph{vertex} in \emph{location} $x\in\eta$ with \emph{(vertex) mark} $s\in(0,1)$. We denote by $\G=\G_{\phi,\eta}=({V}(\G),{E}(\G))$ the random geometric graph obtained by first choosing $V(\G)=\eta'$ and then, conditionally on $\eta'$, generating the edge set $E(\G)$ by adding unoriented edges between $\x,\y\in\eta'$ independently with probability
\[
1-\e^{-\phi(\x\y)},\quad \x\y\in(\eta')^{[2]},
\]
where $A^{[2]}=\{ B\subset A : |B|=2 \}$ for any set $A$ and $\phi(\x\y)$ is the \emph{connection function} of the model. Note that, we always write $\x\y$ for the set $\{\x,\y\}$, when we refer to edges.

\medskip

We focus on the spatially homogeneous case in which 
\[
\phi(\x\y)=\varphi(s,t,|x-y|),\quad \x=(x,s),\y=(y,t),
\]
where $\varphi:(0,1)\times(0,1)\times (0,\infty)$ is a function of marks and mutual vertex distance only (since we consider unoriented edges, this requires $\varphi$ to be symmetric in the first two arguments). Here, $|\cdot|$ denotes Euclidean distance, but all our results remain true for any other norm on $\R^d$. We only consider $\varphi$ which are non-increasing in each argument. Together with a corresponding assumption on $\eta$, this ensures that the weight-dependent random connection model has non-negative correlations, see Section~\ref{sec:Results} below. We further assume the tail integrability condition
\[
\int_{(0,1)}\int_{(0,1)}\int_{\{z\in\R^d:|z|>1\}}
\varphi(s,t,|z|)\di z\di s\di t<\infty,
\]
which ensures local finiteness of $\G_{\eta,\phi}$. No integrability at the origin is required: the underlying point set is locally finite, and connection probabilities are bounded by $1$. A large variety of previously studied percolation models can be obtained as instances of $\G_{\eta,\phi}$ by a suitable choice of $\eta$ and $\phi$. Some of the most important ones are
\begin{itemize}
	\item classical i.i.d.\ Bernoulli($p$) bond percolation: $\eta=\Z^d$, $$\varphi(s,t,r)=-\log(1-p)\1\{r=1\};$$
	\item \emph{long-range percolation} \cite{NewmaSchul86}: $\eta=\Z^d$, $\varphi(s,t,r)=\beta r^{-\delta d}$, for $\beta>0$;
	\item \emph{scale-free percolation} \cite{DeijfHofstHoogh13}: $\eta=\Z^d$, $\varphi(s,t,r)=\beta s^{-\gamma}t^{-\gamma} r^{-\delta d}$, for $\beta>0$ and suitably chosen exponents $\gamma,\delta$;
	\item the \emph{weight-dependent random connection model} \cite{GracaHeydeMoencMoert22}: take a symmetric function $g:(0,1)^2\to(0,\infty)$ and a non-increasing, integrable function $\rho:(0,\infty)\to[0,\infty)$, let
	\begin{equation}\label{eq:varphiviagrho}
		\varphi(s,t,r)=\rho(g(s,t)r^{d})\quad s,t\in(0,1), r\in(0,\infty),
	\end{equation}
	and let $\eta$ be a Poisson point process.
\end{itemize}
Note, however, that the results in this article pertain solely to genuine long-range models in which no upper bound on the length of potential edges exists. More precisely, our goal in the present article is to analyse the \emph{weak decay regime} of inhomogeneous long-range percolation, i.e.\ the regime in which $\delta_{\mathsf{eff}}<2$, where $\delta_{\mathsf{eff}}$ is the (dimension free) exponent of decay of the probability of a long-range connection between large clusters
\[\delta_\mathsf{eff}:=-\lim_{r\to\infty}{\log\left[\int_{r^{-d}}^{1}\int_{r^{-d}}^{1} \varphi(s,t,r) \di s\di t\right]}(d\log r)^{-1}.\]
We discuss this quantity (or rather a closely related one) formally in Section~\ref{sec:Results}, and assume for the moment that it is well-defined. Its role is akin to that of decay exponent $\delta$ in classical long-range percolation (see the examples above). The significance of $\delta_{\mathsf{eff}}$ had been conjectured in \cite{GracaHeydeMoencMoert22} and its importance for the existence of an infinite cluster for one-dimensional models was established in \cite{GracaLuchtMonch22}. In \cite{JahneLuecht23} it was shown that $\delta_{\mathsf{eff}}>2$ implies the existence of a subcritical phase in any dimension, even for models in which connection probabilities between vertices are allowed to depend on other vertices in their spatial vicinity.

\medskip

To give an intuition of how $\delta_{\mathsf{eff}}$ influences the structure of $\G$, assume that $\varphi$ is given via \eqref{eq:varphiviagrho}, for some kernel $g$ and $\rho(x)\asymp x^{-\delta}$. If the kernel $g$ is bounded away from $0$, then our model is merely an inhomogeneous perturbation of classical long-range percolation, for which it is well-known, see e.g.\ \cite{NewmaSchul86,Berge02}, that if $\rho$ decays weakly, namely if $\delta=\delta_{\mathsf{eff}}<2$, then the model feels little of the geometry of the embedding space $\R^d$ and behaves very unlike nearest neighbour percolation on $\Z^d$ and more like a short-range model in high dimensions in some aspects. As was already observed upon the invention of the model \cite{NewmaSchul86}, this is can be derived from the behaviour under rescaling: if one checks whether two large local clusters, each of size $N$, say, are connected directly by a long edge, then the `gain' obtained from independent trials associated with the $N^2-N$ pairs of vertices asymptotically beats the spatial decay of connection probabilities and one finds a connection with high probability if the distance of the clusters is $O(N^{1/d})$. This observation is at the heart of the classical renormalisation group arguments for long-range percolation with $\delta<2$. In particular, $\delta$ moderates both the inter-point connection probabilities and the inter-cluster connection probabilities at large scales, cf.\cite[Lemma 2.4]{Berge02}. However, as was first observed in \cite{GracaHeydeMoencMoert22}, this is not true in the general weight-dependent random connection model: if the kernel $g$ decays sharply at $0$, then the inter-cluster connection probabilities also depend on $g$ and the regime in which. Thus, the regime in which the model easily overcomes the geometric restrictions of the embedding space cannot be found by looking at $\delta$ alone. Instead, one needs to consider the derived exponent $\delta_{\mathsf{eff}}\leq \delta$ depending both on $\delta$ and $g$ and which naturally appears in renormalisation arguments, see \cite{GracaLuchtMonch22}. 
 
 \medskip
 
Below, we establish that $\delta_{\mathsf{eff}}<2$ is sufficient to imply a number of important results which where established for homogeneous long range percolation with $\delta<2$ in \cite{Berge02}. Namely, under varying assumptions on the underlying vertex locations $\eta$, we prove
\begin{itemize}
	\item an asymptotic density result for local clusters of sublinear size (Theorem~\ref{thm:sublinearclusters}),
	\item continuity of the bond percolation function for $\G$ in dimensions $2$ and above (Theorem~\ref{thm:continuity}),
	\item a robustness result for the infinite cluster und removal of long edges in dimensions $2$ and above (Theorem~\ref{thm:truncation}),
	\item transience of the infinite cluster (Theorem~\ref{thm:transience}).
\end{itemize}
There are several technical challenges that we have to overcome which complicate the analysis of the inhomogeneous model compared to long-range percolation. The most crucial one is the presence of additional strong dependencies induced by the vertex marks, which prevents the use of a number of well-established tools for i.i.d.\ (long-range) percolation. Another severe drawback is, as discussed above, that the inhomogeneity influences the scaling behaviour of the models: Coarse-graining a homogeneous long-range percolation model yields another homogeneous long-range percolation model, whereas coarse-graining $\G_{\eta,\phi}$ does not lead to a model that can be readily related to some suitable $\G_{\eta',\phi'}$. The solution of the first problem can be considered as the main contribution of this work on a technical level: We establish renormalisation techniques akin to those in \cite{Berge02} that rely solely on non-negative correlations instead of independence. Our renormalisation relies on a careful use of the FKG-inequality for certain events that are increasing with respect to a sub-$\sigma$-field but not with respect to the full configuration. To the author's knowledge, this is the first use of an argument of this type in the context of percolation; in the field of stochastic processes a technique that is somewhat similar in spirit was used in \cite{Aurz18}. In particular, our proofs are novel even for homogeneous long-range percolation and some of our results are also new for this special case, namely whenever $\eta$ is not Poisson, deterministic or an i.i.d.\ percolated lattice.

Unfortunately, we were not able to overcome the second challenge mentioned above in a similarly comprehensive manner and this is partly reflected in our main results -- most notably we were not able to show that the bond percolation function is continuous throughout the whole weak decay regime if $d=1$, cf.\ Remark~\ref{rem:rem25} .

%



\paragraph{Notation.}
Throughout the article, we use the Landau symbols $f(x)= O(g(x)), f(x)= o(g(x))$ and write $f(x)\asymp g(x)$ if both $f(x)=O(g(x))$ and $g(x)=O(f(x))$. We use $f(x)\sim g(x)$ to denote the stronger statement that $f(x)/g(x)$ converges to $1$.
\paragraph{Overview of the paper.}
In the next section, we provide a formal construction of our model, present our main results and discuss them in more detail. Section~\ref{sec:PercolationFB} contain the proof of Theorem~\ref{thm:sublinearclusters}, which forms the basis of all other main results. Transience of the infinite cluster is obtained in Section~\ref{sec:transience} and the remaining results are proved in Section~\ref{sec:continuity}.

\section{Model definition and main results}\label{sec:Results}
Before formulating our main results, we provide a rigorous construction of the model and definitions of the key quantities and notions involved. 
\paragraph{Construction from a doubly marked point process.} Among our fundamental assumptions are that $\G$ has a unique (if any) infinite component, and that its distribution is the same everywhere in space. We begin by discussing which vertex location sets $\eta$ fall within our framework. Let $\Lambda\subset\R^d$ be either $\Z^d$ or $\R^d$ and let $\eta$ denote a simple point process\footnote{If $\Lambda$ is discrete, then a `point process' on $\Lambda$ is just a random subset.} of finite intensity on $\Lambda$, which is stationary and ergodic under $\P$ with respect to the natural group of shifts $(T_x)_{x\in\lambda}, T_x(y)=x+y$ for all $y\in \Lambda$, associated with $\Lambda$.
\begin{remark}
The choice $\Lambda=\Z^d$ is the most natural one for the discrete set up. However, we do not use any symmetry properties specific to $\Z^d$. Our results are based on renormalisation arguments which use half cubes of the form $(-a,a]^d$ and their translates. All our results remain valid, if one chooses $\Lambda=\{Bz,z\in\Z^d\}$, where $B$ is some non-singular $d\times d$-matrix and replaces the cubes by the corresponding parallelepiped. This changes a few constants appearing below relating volumes and distances but does not alter the content of the theorems. The same applies of course to adapting the norm $|\cdot|$ to $\Lambda$ -- it usually more natural to work with the corresponding lattice distance on $\Lambda$ instead of Euclidean distance.
\end{remark}

The canonical examples for $\eta$ are a homogeneous Poisson process and i.i.d.\ $\operatorname{Bernoulli}(p)$ percolation on $\Z^d$ with $p\in (0,1]$. For simplicity, we assume that $\E\eta((-1/2,1/2]^d)=1$, this can always be achieved by a straightforward rescaling of the ambient space. Although some parts of our considerations are valid under the sole assumptions of ergodicity and positive correlations on $\eta$ (see the paragraph below), some of our main results require a stronger control on dependencies. We say that $\eta$ has \emph{finite range}, if there exists some number $K$ such that $\eta(A)$ and $\eta(B)$ are independent, if $A,B\subset\R^d$ are at distance further than $K$ of each other. Often, it is convenient to view $\eta$ from a typical point, hence we frequently work with the Palm version $\eta_0$ of $\eta$ that has a point at the origin $0\in\R^d$. Note that $\eta$ is translation invariant under shifts of $\Lambda$ if and only if, under $\P_0$, $\eta$ is invariant under shifting the origin into another typical point of $\eta$, see \cite{Thori99}. Our model is now constructed as a deterministic functional of the points of $\eta$, and two independent i.i.d.\ sequences of edge and vertex marks. Let $X_1,X_2,\dots$ denote an enumeration of $\eta$ and let $\mathcal{T}=\{T_j: j\in \Z\}$ be a family of i.i.d.\ random variables distributed uniformly on $(0,1)$ independent of $\eta$. Set
\[\eta'=\{\X_j=(X_j,T_j)\in\eta_0\times\mathcal{T}: j\in\Z \text{ such that }X_k<X_\ell \text{ for } k<\ell\},\]
hence, $\eta'$ is a point process on $\Lambda\times(0,1)$ with unit intensity. Let further $\mathcal{V}=\{V_{i,j}: i<j\in\Z\}$ be a second family of i.i.d.\ Uniform$(0,1)$ random variables, independent of $\eta'$, that we call \emph{edge marks}, which we assign to the elements of $(\eta')^{[2]}$. We denote the point and edge marked process by $\xi$. For given $\varphi$, the graph $\G$ is now deterministically constructed from $\xi$ as the graph with vertex set $\eta$ and edge set
\[\Big\{\X_i\X_j: V_{i,j}\leq 1-\textup{e}^{-\varphi(T_i,T_j,|X_i-X_j|)}\Big\}.\]
Palm versions $\eta'_0, \xi_0$ and $\G_0$ of $\eta',\xi$, and $\G $, respectively, are obtained by replacing $\eta$ in the above construction by $\eta_0$. In the remainder of the paper, the enumeration of the locations plays no role and we usually denote vertices $\x=(x,s)\in\eta'$ as in the introduction and occasionally write $s_{x}$ for the mark of vertex $\x$ in location $x$, and $V_{\x\y}$ for the edge mark associated with the pair $\x\y\in(\eta')^{[2]}$.

Finally, to assure that there is at most $1$ infinite component in $\G$, $\eta$ and $\phi$ should satisfy a suitable `finite energy property', c.f.\ \cite{GandolKeaneNewma92,BurtoMeest93,Meest94} and we shall always assume tacitly that this is the case.

\paragraph{Monotonicity and positive correlation.} 
The inverse vertex mark $s_{x}^{-1}$ can be viewed as weight or fitness of the vertex in location $x\in\eta$, (giving the weight-dependent random connection model its name), the likelihood of connections should be increasing with weight and proximity. Our arguments heavily use the weak FKG-property and to obtain non-negative correlations, we require $\varphi$ to be decreasing in all three arguments. Formally, an \emph{increasing map} of a doubly marked point configuration $\xi$ does not decrease if either
\begin{itemize}
	\item vertices are added to $\eta'$,
	\item vertex marks are decreased,
	\item edge marks are decreased.
\end{itemize}
In particular, if we interpret $\G$ as a map on marked point configurations, it is increasing in the above sense for the canonical partial order on random geometric graphs. An increasing event $E\in \sigma(\xi)$ is such that $\1_E$ is an increasing functional of $\xi$. We require $\G$ to satisfy the weak FKG-property, i.e.\ we have that
\[
\E f(\xi)g(\xi) \geq \E f(\xi)\E g(\xi)
\]
for any increasing functionals $f,g$ on configurations $\xi$. A sufficient condition for this is that $\varphi(\cdot,\cdot,\cdot)$ is non-increasing in all $3$ arguments and that $\eta$ has the weak FKG-property, i.e.\
\[
\E f'(\eta)g'(\eta) \geq \E f'(\eta)\E' g(\eta),
\]
where $f',g'$ are non-decreasing functionals of point configurations under addition of points, c.f.\ \cite{HeydeHofstLastMatzk19,GracaHeydeMoencMoert22} for related constructions.
\begin{remark}
Note that the stated conditions on $\varphi$ and $\eta$ suffice, because the edge and vertex marks are added in an i.i.d.\ fashion. However, the monotonicity assumption on $\varphi$ can be relaxed, e.g.\ if $\eta$ is a Poisson process. Then increasing the intensity of $\eta$ can always be realised by adding another independent Poisson process and thus always increases the resulting graph $\G$, even if $\rho$ is not monotone. On the other hand increasing intensity and contracting space, i.e.\ reducing all inter-location distances, are equivalent. A different direction in which our setup can be generalised is to weaken the requirement of positive correlation on $\eta$, as long as the marks remain independent of $\eta$, since most calculations only require that the model has positive correlations conditionally on $\eta$. Similarly, we believe that our techniques can be adapted without much effort to certain situations in which edge and vertex marks are weakly dependent upon each other or even upon $\eta$, as long as vertex marks and edge marks remain positively correlated. We have not attempted to strive for the most general results in this respect, since the main motivation for our model was to cover all Poisson and lattice based models with i.i.d.\ marks that have so far been treated in the literature in a unified setting. A model with strong positive correlations that is not covered by our approach but might be amenable to certain techniques from the present paper is the \emph{spatial preferential attachment model}\cite{JacobMorte15,JacobMorte17}.
\end{remark}

\paragraph{Weak decay regime.}
We now give a precise definition of the exponent $\delta_{\mathsf{eff}}$ discussed in the introduction. For any $\mu\in[0,1)$ we may set
\[
\deff(\mu):=-\liminf_{r\to\infty}{\log\left[\int_{r^{d(\mu-1)}}^{1-r^{d(\mu-1)}}\int_{r^{d(\mu-1)}}^{1-r^{d(\mu-1)}} \varphi(s,t,r) \di s\di t\right]}(d\log r)^{-1}
\]
By monotonicity, the limit $\deff(0+)=\lim_{\mu\downarrow 0}\deff(\mu)$ exists and we say that $\varphi$ is weakly decaying, if
\[
\deff(0+)<2.
\] 
The standard situation is that both $\deff(0+)=\deff(0)$ and that the $\liminf$ in the definition of $\deff(0)$ can be replaced by an actual limit. In this case, $\deff(0)$ coincides with the exponent $\delta_{\mathsf{eff}}$ discussed in the Section~\ref{sec:intro}. In particular, this is the case in the homogeneous case in which $\varphi$ can be represented as a function $\rho(\cdot)$ of inter-location distance only that satisfies $\rho(z)\asymp z^{-\delta}$. There, we see immediately that $\deff(0)=\deff(0+)=\delta$, cf.\ \cite{GracaLuchtMonch22}.

\paragraph{Percolation.} For $\x\in\eta'$, we write $\C_x$ for the connected component of $\x=(x,s)$ in $\G$. More generally, if $\cG$ is any stationary ergodic geometric random graph, we write $\C_x(\cG)$ for the (possibly empty if $x$ is no point of $\eta$) connected component $\cG$ of the vertex located in $x$. The maximal component of $\cG$ is denoted by $\C_{\mathsf{max}}(\cG)$ or $\C_{\infty}(\cG)$ if it happens to be of infinite size (recall that we exclusively consider situations in which $\C_{\infty}(\cG)$ is unique). Set now
\[\theta_{\cG}=\P(|\C_0(\cG_0)|=\infty),\]
where $\cG_0$ is the Palm version of $\cG$ (the latter always exists, since $V(\cG)$ must be distributionally invariant under shifts along $\Lambda$ by stationarity). By ergodicity, we have that $\theta_{\cG}\in[0,1]$ is constant and corresponds to the density of the infinite component.

 Our first and most general result localises the existence of an infinite cluster. We use the notation $\G[\Gamma^n]$ to denote the subgraph induced in $\G$ by the vertices located in $\Gamma^n=(-n/2,n/2]^d$.
\begin{theorem}[Local clusters of sublinear size are asymptotically dense]\label{thm:sublinearclusters}
	Let $\G$ denote an instance of inhomogeneous long-range percolation on a stationary, ergodic and positively correlated point set $\eta$ such that $\deff(0+)<2$. If $\theta_\G>0$, then for every $\lambda\in(0,1)$, we have
	\[
	\lim_{n\to\infty}\P\big(\big|\C_{\mathsf{max}}(\G[\Gamma^n])\big|> n^{\lambda d}\big)=1.
	\]
\end{theorem}

It stands to reason, that the assertion of Theorem~\ref{thm:sublinearclusters} can be improved to a stretched exponential bound on the probability of existence of a local cluster of \emph{linear size}, at least if we assume independence for $\eta$. We plan to address this in future work. The corresponding result for classical long-range percolation was established in \cite[Theorem 3.2]{Bisku04}.
%
%

Set now
\[\theta(p)= \theta_{\G^p},\quad p\in[0,1],\]
where $\G^p$ is obtained from $\G$ by independent Bernoulli bond percolation with retention probability $p$. When no additional percolation is involved we also write $\theta=\theta(1)$ for the density of the infinite cluster in $\G$. The following result states that $\theta(p)$ is a continuous function of $p$ in two or more dimensions as long as we remain in the weak decay regime.
\begin{theorem}[Continuity of bond percolation function]\label{thm:continuity}
If $d\geq 2$, $\G$ is locally finite, $\eta$ has finite range and $\bar\delta_{\mathsf{eff}}(0+)<2$, then
\[
p\mapsto\theta(p)
\]
is a continuous function on $[0,1]$.
\end{theorem}
Theorem~\ref{thm:continuity} extends \cite[Thm 1.5]{Berge02}, \cite[Cor. 4]{DepreHazraWuthr15} and \cite[Thm. 3.3]{DepreWuthr19} for $d\geq 2$. However, note that all three previous results correspond to the case in which $\deff(0+)$ trivially coincides with the a priori spatial decay exponent $\delta$ (in our notation based on the WDRCM). In particular, in the special case of \emph{scale-free percolation} \cite{DeijfHofstHoogh13,DepreHazraWuthr15,DepreWuthr19}, $\deff(0+)<2\leq \delta$ precisely if the critical threshold is $0$, in which case Theorem~\ref{thm:continuity} is standard. 
\begin{remark}\label{rem:rem25}
The conclusion of Theorem~\ref{thm:continuity} does not include $d=1$, unless one has that $\bar\delta_{\mathsf{eff}}(0+)$ coincides with the spatial decay exponent of the connection probabilities. In that case, our model is only a weak perturbation of homogeneous long-range percolation and the corresponding result is a minor extension of \cite[Theorem 1.5]{Berge02} that can be obtained with the help of our Theorem~\ref{thm:sublinearclusters} and the block-renormalisation techniques of \cite{Berge02}. The difficulty in one dimension is rooted in the scaling behaviour for truly inhomogeneous model instances: in general, block-renormalised versions of the model behave quite differently to the original model. For homogeneous long-range percolation the opposite is true, as discussed in the introduction, cf.\ \cite[Lemma 2.4]{Berge02}. In our renormalisation arguments, the tool used to connect large clusters is Lemma~\ref{lem:preciseconnection} below, which is only effective for clusters close enough to each other. Therefore our proof of Theorem~\ref{thm:continuity} relies on comparison with supercritical nearest-neighbour models, whereas the $d=1$ case would require a comparison with a suitable supercritical long-range model. To establish such a comparison in dimension $1$ when $\delta_{\mathsf{eff}}<2$ is outside the scope of this article but the subject of ongoing research. The related fact that percolation may occur in $d=1$ only inside the weak decay regime (or possibly at its boundary) was established in \cite{GracaLuchtMonch22} for the weight-dependent random connection model. However, the techniques used there are not strong enough to make assertions about the behaviour of $\theta(p)$ near the critical value.
\end{remark}

 It follows from the proof of Theorem~\ref{thm:continuity}, that percolation in $\G$ is also robust under edge truncation, which we formulate as our next theorem. Denote by $\G^{\{\ell\}}$ the graph obtained from $\G$ by removing all edges longer than $\ell>0.$
\begin{theorem}[Truncation property]\label{thm:truncation}
If $d\geq 2$, $\eta$ has finite range $\bar\delta_{\mathsf{eff}}<2$ and $\G$ percolates, then
\[
\lim_{\ell\to\infty}\theta_{\G^{\{\ell\}}}>0.
\]
\end{theorem}
In the strong decay regime, the truncation property was recently established for classical long range percolation in \cite{Baum23}. Note that both Theorem~\ref{thm:sublinearclusters} and Theorem~\ref{thm:truncation} provide `locality'-type statements for percolation, i.e.\ if $\G$ percolates and $\G_n\to\G$ locally, then the $\G_n$ will percolate eventually. In Theorem~\ref{thm:truncation}, the approximating sequence is obtained by graphs of bounded edge length and in Theorem~\ref{thm:sublinearclusters} by the model restricted to large cubes. Note that in the latter case, `percolation' on the approximating graph corresponds to the existence of a cluster that grows with the size of the box.
\paragraph{Transience.}
A connected loop-free multigraph $\mathcal{G}=(V(\mathcal{G}),E(\mathcal{G}))$ together with a \emph{conductance} function $C\colon E(\mathcal{G})\to (0,\infty)$ is called a \emph{network}. Note that we may always view $C$ as a function defined on $V(\mathcal{G})^{[2]}$ setting $C(xy)=0$ for potential edges $xy\notin E(G)$. The random walk $Y=(Y_i)_{i\geq 0}$ on $(\mathcal{G},C)$ is obtained by reweighing the transition probabilities of simple random walk on $\mathcal{G}$ according to $C$, i.e.\ the walker chooses their way with probabilities proportional the sum of the conductances on the edges incident to their current position. In particular, we obtain simple random walk 
on~$G$ as a special case, {if $C$ is constant}. We only consider locally finite networks, i.e.\ 
$$\pi(x):=\sum_{y\in V(\mathcal{G}): y \text{ incident to }x}C(xy)<\infty\quad \text{for all }x\in V(\mathcal{G}).$$ 
Note that $\pi$ is an invariant measure for $Y$. Let further $\mathsf{P}^\mathcal{G}(v\to Z)$ denote the probability that $Y$ visits $Z\subset V(\mathcal{G})$ before returning to $v\in V(\mathcal{G})$ when started in $Y_0=v\in V(\mathcal{G}).$ Now define the \emph{effective conductance} between $v\in V(\mathcal{G})$ and $Z\subset V(\mathcal{G})$ as
\[
\mho(v,Z)=\mho^\mathcal{G}(v,Z)=\pi(v)\mathsf{P}^\mathcal{G}(v\to Z),
\]
for finite $\mathcal{G}$ and then extend the notion to infinite graphs via a limiting procedure. In particular, by identifying all vertices at graph distance further than $n$ from $v\in V(\mathcal{G})$ with one vertex $z_n$ (whilst removing any loops and keeping multiple edges with their conductances) we obtain a sequence of finite networks $(\mathcal{G}_n,C_n)$. Moreover, the limit
\[
\mho^\mathcal{G}(v,\infty)=\lim_{n\to\infty}\mho^{\mathcal{G}_n}(v,z_n)\in[0,\infty)\quad\mbox{ for } v\in V(\mathcal{G}),
\]
is well-defined. We say that $Y$ (or $\mathcal{G}$ if $Y$ is simple random walk on $\mathcal{G}$) is transient if
	\[
	\mho(v,\infty)>0 \quad\text{ for some }v\in V(\mathcal{G}).
	\]
	Moreover, if $\mho(v,\infty)>0$ for some $v\in V(\mathcal{G})$, then $\mho(v,\infty)>0$ for all $v\in V(\mathcal{G})$.
\begin{theorem}[Transience of infinite cluster]\label{thm:transience}
Let $\varphi$ be such that $\bar\delta_{\mathsf{eff}}<2$ and let $\eta$ be either a Poisson process or an i.i.d.\ percolated version of $\Z^d$. Then, if an infinite component in $\G$ exits, it is almost surely transient.
\end{theorem}
The proof of Theorem~\ref{thm:transience} is given in Section~\ref{sec:transience}. Theorem~\ref{thm:transience} is strictly stronger than the previous transience results in \cite{Berge02,HeydeHulshJorri17,GracaHeydeMoencMoert22} and in particular establishes the recurrence-transience transition conjectured for the two-dimensional \emph{soft Boolean model} in \cite{GracaHeydeMoencMoert22}. In $d\geq 3$ transience should of course also hold outside the weak decay regime, but this is difficult to establish for the same reason as the corresponding truncation result.


\section{Percolation in finite boxes}\label{sec:PercolationFB}
Throughout the following sections, we work repeatedly with the collections of half-open cubes
\[
\mathfrak{C}(m):=\{x+[-m/2,m/2)^d,\, x\in m\Z^d \},\quad m\in\N.
\]
We write $\Gamma^m(x)=x+[-m/2,m/2)^d$, $x\in m\Z^d$ for the cube of side length $m$ centred at $x\in m \Z^d$ and write $\Gamma^m$ for $\Gamma^m(0).$ For any bounded domain $\Lambda\subset \R^d$, we define the \emph{k-neighbourhood} of $\Lambda$ as $$\Delta_k\Lambda=\{x\in\R^d: \inf_{y\in\Lambda} |x-y|_{\infty}\leq k \}.$$
If $\cG=\big(\mathcal{V}(\cG),\mathcal{E}(\cG)\big)$ is any random geometric graph and $D\subset\R^d$ is some bounded domain, we write $\cG[D]$ for the subgraph of $\cG$ induced by vertices located in $D$. 

To prove Theorem~\ref{thm:sublinearclusters}, we first establish some auxiliary results and develop an improved version of the renormalisation approach used in \cite{Berge02} to study homogeneous long-range percolation. Let us begin by setting up some notation. We say that a finite collection $M$ of numbers in $(0,1)$ is \emph{$\mu$-regular}, for $\mu\in(0,1/2)$, if for all $i\in I(\mu,M):= \{1,\dots,\lfloor |M|^{1-\mu} \rfloor\}$ it holds that 
\[
N_i(M):=\sum_{S\in M}\1\left\{S\leq \tfrac i{|I(\mu,M)|}\right\}\geq \frac{|M|}{2}\,\frac{i}{|I(\mu,M)|}.
\]
A vertex set $V\subset\eta'$ is called \emph{$(\mu,k)$-regular}, if there exist $(x_1,s_1),\dots,(x_k,s_k)\in V$, such that $\{s_1,\dots,s_k\}$ is $\mu$-regular. 
\begin{lemma}\label{lem:regular}
Fix $\mu\in(0,1/2)$. Any finite collection $M$ of i.i.d.\ Uniform$(0,1)$ random variables is $\mu$-regular with probability exceeding
\[
1-|M|^{1-\mu}\e^{-|M|^\mu/8}.
\]
\end{lemma}
\begin{proof} Let $n=|M|$. We have
\[
\E [N_i(M)] =\frac{n i}{|I(\mu,M)|}, \quad i\in I(\mu,M), 
\]
and, by Bernstein's inequality,
\begin{align*}
	\P(\exists i :\; N_i(M)<ni/(2|I(\mu,M)|)) & \leq \sum_{i=1}^{|I(\mu,M)|}\P(N_i(M)<\E N_i(M)/2) \\
	& \leq \sum_{i=1}^{|I(\mu,M)|}\e^{-\frac{(n i/|I(\mu,M)|)^2}{8n i/|I(\mu,M)| }}= \sum_{i=1}^{|I(\mu,M)|}\e^{-\frac{i n}{8|I(\mu,M)|}}.
\end{align*}
By definition of $I(\mu,M)$, we have $|I(\mu,M)|\leq n^{1-\mu}$ and the claimed bound follows.
\end{proof}
The purpose of $\mu$-regularity is to obtain lower bounds on connection probabilities for large vertex sets that depend solely on their size and distance of each other. We say that two disjoint vertex sets $V_1,V_2\subset\eta'$ are \emph{adjacent} (in $\G$), denoted by $V_1 {\leftrightarrow} V_2$, if there exist $\x\in V_1,\y\in V_2$ with $\x\y\in E(\G).$
\begin{lemma}\label{lem:preciseconnection}
	There exists a constant $C=C(g,\rho)>0$ and for any $\mu\in(0,1/2)$ some $v^\ast(\mu)< \infty$ such that for all $v\geq v^\ast(\mu)$ and any disjoint pair $V_1,V_2\subset\eta'$ of $(\mu,v)$-regular vertex sets that satisfy $\textup{diam}(\{x:(x,s)\in V_1\cup V_2\})\leq D$, we have the uniform deterministic bound
	\[
	\P(V_1 {\leftrightarrow} V_2|\eta')\geq1-\exp\left( - C\, v^2 \int_{[v^{-(1-\mu)},1-v^{-(1-\mu)}]^2}\varphi(s,t,D)\, \textup{d}(s,t)\right ).\]
\end{lemma}
\begin{remark}
$\mu$-regularity of a vertex set is solely a property of the i.i.d.\ vertex marks and, conditionally on $\eta'$, the event $V_1 {\leftrightarrow} V_2$ is measurable with respect to the edge marks of edges joining $V_1$ and $V_2$ only. Hence, Lemma~\ref{lem:regular} yields a large deviation bound for untypical behaviour of vertex marks and Lemma~\ref{lem:preciseconnection} is essentially a large deviation bound for the i.i.d.\ sequence of edge marks, given that the vertex marks involved show typical behaviour.
\end{remark}
\begin{proof}[{Proof of Lemma~\ref{lem:preciseconnection}}]
	Let $V_i, i=1,2$ denote vertex sets of size $v$ such that all locations of vertices in $V_1\cup V_2$ are within distance $D$ of each other and such that $V_1$ and $V_2$ have $\mu$-regular marks. Let further $F_i$ be the empirical distribution function of the vertex marks corresponding to $V_i$, $i=1,2$, set $n=|V_1|$, $r=|I(\mu,V_1)|$ and denote by $M_1=\{s_1,\dots,s_n\}$ the vertex marks corresponding to $V_1$. We have, for $t\in[0,1]$,
\begin{align*}
	n F_1(t) & = \sum_{j=1}^r \sum_{i=1}^n \1\{ s_i\leq t \}\1\left\{ t \in \left(\tfrac{j-1}{r},\tfrac{j}{r} \right]\right\}
	\geq \sum_{j=1}^r N_{j-1}(M_1)\1\left\{ t \in \left(\tfrac{j-1}{r},\tfrac{j}{r} \right]\right\}\\
	& = \sum_{k=1}^r  (N_k(M_1)-N_{k-1}(M_1))\sum_{j=k+1}^r \1\left\{ t \in \left(\tfrac{j-1}{r},\tfrac{j}{r} \right]\right\} = N_{\lfloor tr \rfloor}(M_1).
\end{align*}
Since $M_1$ is $\mu$-regular, this implies that
\[
nF_1(t)\geq \frac{n}{2}\frac{\lfloor tr \rfloor}{r}\geq \frac{n}{2}\frac{tr -1}{r}=\frac{n}{2}(t-1/r).
\]
A similar argument holds for $F_2$ and it follows that, for $|V_1|, |V_2|$ sufficiently large,
\begin{equation}\label{eq:distestimate}
	F_1(t)\geq \frac{1}{3}(t-|V_1|^{-(1-\mu)}) \text{ and } F_2(t)\geq \frac{1}{3}(t-|V_2|^{-(1-\mu)}). 
\end{equation}
Now note that $\x_i=(x_i,t_i)\in V_1$ and $\x_j=(x_j,t_j)\in V_2$ are always connected if their corresponding edge mark satisfies
\begin{equation}\label{eq:aux}
	V_{\x_i\x_j}\leq 1-\e^{-\varphi(t_i,t_j,D)},
\end{equation} which can be evaluated independently of the exact spatial positions. Since the edge mark collection $\{V_{\x_i\x_j}; \x_i,\x_j\}$ is i.i.d.\ and independent of $\eta'$ we have for the number $\Sigma$ of edges between $V_1,V_2$
\[
\P(\Sigma=0|\eta')\leq \prod_{{(x_i,t_i)\in V_1, }\atop{(x_j,t_j)\in V_2}}\e^{-\varphi(t_i,t_j,D)}=\exp\Big(-\sum_{{(x_i,t_i)\in V_1, }\atop{(x_j,t_j)\in V_2}}\varphi(t_i,t_j,D)\Big).
\]
It now follows from \eqref{eq:distestimate} and the non-negativity of distribution functions that for $h(s,t)=\varphi(s,t,D), (s,t)\in(0,1)^2$, \begin{align*}
	\sum_{{(x_i,s_i)\in V_1,}\atop{ (x_j,s_j)\in V_2}} h(s_i,s_j) & =|V_1||V_2|\int_0^1\int_0^1 h(s,t)\, \textup{d} F_1(s)\textup{d} F_2(t) \\ & \geq  \frac{|V_1||V_2|}{9}\int_{V_1^{-(1-\mu)}}^{1-|V_1|^{-(1-\mu)}}\int_{|V_2|^{-(1-\mu)}}^{1-|V_2|^{-(1-\mu)}} h(s,t)\, \textup{d}s
	\textup{d}t,
\end{align*}
and the assertion of the lemma follows, because the estimate is uniform in the configuration $\eta'$ on the event that $V_1,V_2$ are $\mu$-regular.
\end{proof}
The renormalisation scheme we use to prove Theorem~\ref{thm:sublinearclusters} requires a number of interdependent parameters, which we now introduce. We first choose a sequence of density parameters $(\varrho_n)$ such that $\varrho_n<1/4$ for all $n\in\N$ and
\begin{equation}\label{eq:varrhodef}
	\lim_{n\to\infty}\varrho_n n^2\in(1,\infty).
\end{equation}
In fact, the precise polynomial decay of $\varrho_n$ is not important as long as 
\begin{equation*}
	\sum_{n=1}^\infty \varrho_n <\infty.
\end{equation*}
Now fix $\mu^\ast=\mu^\ast(\varphi)\in(0,1/2)$ such that
\begin{equation}\label{eq:mustardef}
	\deff(\mu^\ast)<2,
\end{equation}
which is possible due to our assumption on $\deff(0+)$. Now choose $\nu=\nu(\varphi)$ satisfying
\begin{equation}\label{eq:nudef}
	1<\nu< \begin{cases}\frac{1}{1-\mu^\ast}, & \text{ if }\deff(\mu^\ast)\leq 0,\\
		\frac{1}{1-\mu^\ast} \wedge \frac{2}{\deff(\mu^\ast)}, & \text{ if }\deff(\mu^\ast)\in(0,2),
	\end{cases}
\end{equation}
and let
\begin{equation}\label{eq:mudef}
	\mu=1-\nu(1-\mu^\ast)\in(0,\mu^\ast).
\end{equation}
Finally, we choose $(\sigma_n)$ such that $\sigma_n\in 2\N+1$ for all $n$ and such that for all but finitely many $n$,
\begin{equation*}
	n^{\omega} \leq \sigma_n \leq (1+n^{-2})^{1/d} n^{\omega},
\end{equation*}
where $\omega=\omega(\varphi)$ satisfies
\begin{equation}\label{eq:omegadef1}
	\omega>\frac{2 \nu}{d(\nu-1)}>\frac{2}{d}.
\end{equation}
For the remainder of the section, one should think of the sequences $(\varrho_n),(\sigma_n)$ and the numbers $\nu,\mu$ and $\omega$ as having been fixed. Assume further, that two large integers $k\in\N,\ell\in 2\N$ are given -- we are going to specify these parameters below dependent on the density of the infinite cluster $\theta$ and the auxiliary variable $\lambda$ appearing in the formulation of Theorem~\ref{thm:sublinearclusters}. Define a sequence $(m_n)=(m_n(\ell))$ of lengths via
\[
m_n=\ell \prod_{i=1}^n \sigma_i,\quad n\in\{0,1,2,\dots\}.
\]
To lighten notation, we write $\Gamma_n(x)=\Gamma^{m_n}(x),x\in m_n\Z^d,$ for the \emph{stage-}$n$ cube at $x$, i.e.\ the cube in $\mathfrak{C}(m_n)$ with midpoint $x$. Note that for any $n\in\N$, each stage-$n$ cube can be decomposed into precisely $\sigma_n^d$ stage-$n-1$ cubes, which we call its \emph{subcubes}. 
The \emph{preclusters} of a cube $\Gamma_n(x)$ are maximal subsets (with respect to inclusion) of $\eta'\cap \Gamma_n(x)\times(0,1)$ which are contained in the same connected component of $\G[\Delta_k\Gamma_n(x)]$ (note that their definition depends on $k$).

\medskip

A stage-$0$ cube is a cube $\Gamma_0(x)\in\mathfrak{C}(\ell)$ and said to be \emph{alive} if it contains a precluster
with at least $\lceil\ell^d\theta/2\rceil$ vertices. Similarly, the cardinality thresholds
\[
v_n=\frac{\theta}{2} m_{n}^d \prod_{i=1}^{n}\varrho_i, \quad n\in\{0,1,2,\dots\},
\]
act as lower bounds for the number of vertices in the preclusters at further stages, but the condition for aliveness becomes a little more complex. For $n\geq1$, a stage-$n$ cube $\Gamma_n(x)$ is alive, if
\begin{itemize}
	\item[\textsf{A}($n$)] at least $r_n:=\lceil\varrho_n\sigma_n^d\rceil$ of its subcubes are alive;
	\item[\textsf{B}($n$)] at least $r_n$ of its living subcubes contain a $(\mu,v_{n-1})$-regular precluster;
	\item[\textsf{C}($n$)] there are $(\mu,v_{n-1})$-regular preclusters $\mathcal{C}_1,\dots,\mathcal{C}_{r_n}$, each associated with a different subcube, which are are all mutually adjacent in $\G$.
\end{itemize}
There is some redundancy in defining aliveness via the properties $\textsf{A}(n),\textsf{B}(n)$ and $\textsf{C}(n)$, but this formulation makes it straightforward to relate the definition to probabilities. Note, that the construction ensures that a living stage-$n$ subcube always contains a precluster of size at least $v_n$. Furthermore, the events $$\{\Gamma_n(x) \text{ is alive}\}, \quad n\in\N\cup\{0\},x\in m_n\Z^d,$$ are increasing. The main tool needed to prove Theorem~\ref{thm:sublinearclusters} is the following lemma.
\begin{lemma}\label{lem:alivebound}
There exists some constant $0<C<\infty$ depending only on $\varphi$ and $\ell$ such that
\begin{equation}\label{eq:mainprobbound}
\P(\Gamma_n\text{ is not alive})\leq C \P(\Gamma_0\text{ is not alive}), \quad \text{for all }n\in\N.
\end{equation}	
\end{lemma}	
To establish Lemma~\ref{lem:alivebound}, we proceed in several steps. For $n\in\N\cup\{0\}$, we denote by $\Gamma_n\in\mathfrak{C}(m_n)$ the stage-$n$ cube centred at the origin. We now define the events
\begin{itemize}
	\item $A_n=\{\Gamma_n \text{ is alive}\}, n\in\N_0$,
	\item $A'_n=\{\Gamma_n \text{ has at least }r_n \text{ living subcubes} \}, n\geq 1,$
	\item $B_n=\{\Gamma_n \text{ satisfies condition \textsf{B}($n$)}\}, n\geq 1,$
	\item $C_n=\{\Gamma_n \text{ satisfies condition \textsf{C}($n$)}\}, n\geq 1,$
\end{itemize}
and aim to give lower bounds for their probabilities.
Our first result is a straightforward recursive bound for $\P(A'_n)$ in terms of $\P(A_n)$.
\begin{lemma}\label{lem:a_nbound}
Set $$a_n:=\P((A_n)^\mathsf{c}),\quad n\in\N\cup\{0\}.$$ Then
\[
\P((A'_n)^\mathsf{c})\leq \frac{a_{n-1}}{1-\varrho_n},\quad n\in\N.
\]
\end{lemma}
\begin{proof}
 By translation invariance and Markov's inequality,
\begin{equation*}
\P((A'_n)^\mathsf{c})=\P(\# \{\text{non-alive subcubes of }\Gamma_n\}>(1-\varrho_n)\sigma_n^d) \leq \frac{\P(A_{n-1}^\mathsf{c})}{1-\varrho_n}, \quad n\geq 1.
\end{equation*}
\end{proof}
To obtain further bounds involving the events $B_n$ and $C_n$, we define the \emph{maximal precluster} $\mathcal{C}^{n,k}(x)\subset \Gamma_n(x)\times(0,1)$ of a stage-$n$ cube $\Gamma_n(x)$ to be its precluster of largest cardinality (note that $\mathcal{C}^{n,k}(x)$ may be empty if $\eta\cap\Gamma_n(x)$ is empty), unless there is a tie between several preclusters, in which case the maximal precluster is the (almost surely unique) one amongst them containing the vertex with the smallest mark. This definition only works for almost every configuration $\xi$, we thus set $\mathcal{C}^{n,k}(x)=\emptyset$ on the set of configurations $\xi$ on which there are at least two preclusters of maximal size with the same minimal mark to obtain a well-defined precluster in any case. Analogously, we define $\mathcal{R}^{n,k}(x)\subset \Gamma_n(x)\times(0,1)$ to be the \emph{maximal $(\mu,v_n)$-regular precluster} associated with a stage-$n$ cube $\Gamma_n(x)$.
\begin{remark}
	Note that once $\mathcal{C}^{n,k}(x)$ is non-empty for some configuration $\xi$, it remains non-empty if any vertex mark of a vertex in $\mathcal{C}^{n,k}(x)$ is decreased or if any edge mark of an edge adjacent to $\mathcal{C}^{n,k}(x)$ is decreased. The same is true for $\mathcal{R}^{n,k}(x)$, respectively. This fact is needed to obtain a monotonicity property of the events $E(\cdot),F(\cdot,\cdot)$ defined in Lemma~\ref{lem:decomp}.
\end{remark}
Let now $\Gamma_n[1],\dots,\Gamma_n[\sigma_n^d]$ denote the subcubes of $\Gamma_n, n\in \N$ with corresponding centers $x_n(i)$ and maximal preclusters $\mathcal{C}^{n-1,k}(x_n(i)),\mathcal{R}^{n-1,k}(x_n(i)), 1\leq i\leq\sigma_n^d$.
\begin{lemma}\label{lem:decomp}
For any $n\in\N$, we have
\begin{equation}\label{eq:eventdecomp}
	\begin{aligned}
		(C_n\cap B_n)^{\mathsf{c}}\cap A'_n \subset & \bigcup_{i=1}^{\sigma_n^d} (A_{n-1}(i)\cap E(i)) \\ 
		&\cup \bigcup_{i,j=1}^{\sigma_n^d} B(i)\cap B(j) \cap F(i,j).
	\end{aligned}
\end{equation}
where \begin{itemize}
	\item $A_{n-1}(i)=\{\Gamma_n[i]\text{ is alive}\}, 1\leq i\leq \sigma_n^d$;
	\item $B(i)=\{\Gamma_n[i]\text{ contains a }(\mu,v_{n-1})\text{-regular precluster} \}, 1\leq i\leq \sigma_n^d$;
	\item $E(i)=\{ \mathcal{C}^{n-1,k}(x_n(i)) \text{ is not }(\mu,v_{n-1})\text{-regular}\}, 1\leq i\leq \sigma_n^d$;
	\item $F(i,j)= \{\mathcal{R}^{n-1,k}(x_n(i)) \not\leftrightarrow \mathcal{R}^{n-1,k}(x_n(j))\}, 1\leq i,j\leq \sigma_n^d$.
\end{itemize}
\end{lemma}
\begin{proof}
We have the disjoint decomposition $$ (C_n\cap B_n)^{\mathsf{c}}\cap A'_n = (B_n^{\mathsf{c}}\cap A'_n) \cup (C_n^{\mathsf{c}}\cap B_n \cap A'_n),$$
and on the first event $B_n^{\mathsf{c}}\cap A'_n$, there has to be a living subcube that has no $(\mu,v_{n-1})$-regular precluster, which implies that its largest precluster cannot be $(\mu,v_{n-1})$-regular. The second event satisfies
$$ 
C_n^{\mathsf{c}}\cap B_n \cap A'_n \subset \{\mathcal{R}^{n-1,k}(x_i) \leftrightarrow \mathcal{R}^{n-1,k}(x_j) \text{ for all }i,j\in \mathcal{I}\}^{\mathsf{c}} \cap B_n\cap A'_n,
$$
where $\mathcal{I}$ is the set of indices $i$, such that $\mathcal{R}^{n-1,k}(x_i)\neq \emptyset.$ If $r_n=1$ (which we have not explicitly excluded), then $C_n$ always holds on $B_n\cap A'_n.$ Otherwise, on $B_n\cap A'_n$, we have that $\mathcal{I}$ contains at least $r_n\geq 2$ indices. Hence on $\{\mathcal{R}^{n-1,k}(x_i) \text{ for all }i,j\in \mathcal{I}\}^{\mathsf{c}} \cap B_n\cap A'_n$, there has to be a pair of boxes containing a $(\mu,v_{n-1})$-regular precluster each, but such that the maximal such precluster in either box are not adjacent. Thus, \eqref{eq:eventdecomp} is established.
\end{proof}

The next two lemmas complete the estimates that we need to prove Lemma~\ref{lem:alivebound}. For their proofs we use the following auxiliary subsampling of vertices: Let $\eta$ be given. To each stage-$n$ cube $\Gamma_n(x)\in \mathfrak{C}(m_n)$, we assign a sample $X(x,n)=\{X_1(x,n),\dots,X_{v_n}(x,n)\}$ of \emph{tagged} vertex locations in $\eta\cap\Gamma_n(x)$, chosen uniformly without replacement and such that the collection of point sets $$X=\big\{X(x,n), x\in m_n\Z^d, n\in\{0,1,2,\dots\} \big\}$$ is mutually independent over all distinct choices of indices $(x,n)$ and also the whole collection is independent of all vertex and edge marks. If $\eta\cap\Gamma_n(x)$ contains fewer than $v_n$ vertices, then we set $X(x,n)=\emptyset$. We write $$\X=\big\{\X(x,n),x\in m_n\Z^d, n\in\{0,1,2,\dots\}\big\}$$ for the collection of tagged vertex sets corresponding to the tagged sites in each cube. The configuration $\xi$ augmented by the independent tagging is denoted $\bar{\xi}$ and the induced probability distribution on tagged configurations by $\bar{\P}$.
\begin{lemma}\label{lem:b_nbound}
Let 
\[
b_n= \P\left(\bigcup_{i=1}^{\sigma_n^d} (A_{n-1}(i)\cap E(i))\right),\quad n\in\N,
\]
then
\[b_n\leq \sigma_n^d v_{n-1}^{1-\mu}\e^{-v_{n-1}^\mu/8},\quad n\in\N.
\]
\end{lemma}
\begin{proof}
A simple union bound and translation invariance yield
\begin{equation}\label{eq:simplebnbound}
b_n\leq \sigma_n^d \P(A_{n-1}(1)\cap E(1)), \quad n\in\N,
\end{equation}
Hence it remains to estimate the probability on the right. Fix $n\geq 1$. We define an alternative tagging of vertex locations in $\Gamma_n[1]$ depending on $\eta$ as well as edge and vertex marks. Namely, we set $Y=\emptyset$ on $A_{n-1}(1)^{\mathsf{c}}$ and on $A_{n-1}(1)$, we set $Y=\{Y_1,\dots,Y_{v_{n-1}}\}$, where $Y_1,\dots,Y_{v_{n-1}}$ are chosen uniformly without replacement amongst the vertex locations belonging to the maximal precluster $\mathcal{C}^{n-1,k}(x_n(1))$ of $\Gamma_n(1)$. Let $\bar{\mathbf{P}}_\eta(\cdot):=\bar{\P}(\cdot|\eta)$ and denote the joint distribution of $\xi$ and $Y$ by $\tilde{\P}$ and its conditional version given a fixed point configuration $\eta$ by $\tilde{\mathbf{P}}_\eta$. Note that on $A_{n-1}(1)$, $\eta$ must have at least $v_{n-1}$ points in $\Gamma_n[1]$. It follows from the uniformity of the sample $X(x_n(1),n-1)$ and its independence of edge end vertex marks, that
\begin{equation}\label{eq:conditonaleq}
\begin{aligned}
&\tilde{\mathbf{P}}_\eta((\xi,Y)\in \cdot \;| A_{n-1}(1))\\ 
& = \bar{\mathbf{P}}_\eta\big((\xi,X(x_n(1),n-1))\in \cdot \;| A_{n-1}(1), \X(x_n(1),n-1)\subset \mathcal{C}^{n-1,k}(x_n(1)) \big),
\end{aligned}
\end{equation}
since a uniform sample drawn from a finite set $S$ conditioned to be contained in an independently generated random subset $S'\subset S$ has the same distribution as a uniform sample drawn from $S'$. We have
\begin{equation}\label{eq:withtilde}
\begin{aligned}	
\P(A_{n-1}(1)\cap E(1)|\eta) & = \bar{\mathbf{P}}_\eta(E(1)\cap A_{n-1}(1)) \\
& \leq \bar{\mathbf{P}}_\eta(E(1)| A_{n-1}(1)) = \tilde{{\mathbf{P}}}_\eta(E(1)|A_{n-1}(1)),
\end{aligned}
\end{equation}
where we define the conditional probabilities to equal $0$ if $|\eta\cap\Gamma_n[1]|<v_{n-1}$ and the equalities are due to the fact that the events $E(1),A_{n-1}(1)$ do not involve the tagging at all. However, denoting by $S=\{S_1,\dots,S_{v_{n-1}}\}$ the vertex marks belonging to the tagged vertex locations in $X(x_n(1),n-1)$ and by $T=\{T_1,\dots,T_{v_{n-1}}\}$ the vertex marks belonging to the tagged vertex locations in $Y$, we also have
\begin{equation}\label{eq:barfortilde}
\begin{aligned}
&\tilde{{\mathbf{P}}}_\eta(E(1)|A_{n-1}(1))\leq \tilde{\P}(T \text{ is not $\mu$-regular}|A_{n-1}(1))\\
&= \bar{{\mathbf{P}}}_\eta\big(S \text{ is not $\mu$-regular} | A_{n-1}(1) \cap\{ \X(x_n(1),n-1)\subset \mathcal{C}^{n-1,k}(x_n(1)) \}\big),
\end{aligned}
\end{equation}
by \eqref{eq:conditonaleq}. Yet $S$ is an i.i.d.\ sample of $v_{n-1}$ $\operatorname{Uniform}(0,1)$ random variables under $\bar{\P}(\cdot|\eta)$, whenever $|\eta\cap \Gamma_n(1)|\geq v_{n-1}$. Moreover, we claim that the events $A_{n-1}(1)$ and $D:=\{ \X(x_n(1),n-1)\subset \mathcal{C}^{n-1,k}(x_n(1))\}$ are increasing in $S$ and that the event $\{S \text{ is not $\mu$-regular}\}$ is decreasing in $S$. This is easily seen to be true for $A_{n-1}(1)$ and $E(1)$. The statement for $D$ holds because if $(\xi,X(x_n(1),n-1))$ is such that $\{ \X(x_n(1),n-1)\subset \mathcal{C}^{n-1,k}(x_n(1))\}$, then decreasing any one of the marks $S_i, 1\leq i\leq v_n$, can only increase the maximal precluster $\mathcal{C}^{n-1,k}(x_n(1))$ in size (or decrease the lowest mark, if there is a tie in sizes), in particular this means that none of the tagged vertices can leave $\mathcal{C}^{n-1,k}(x_n(1))$ if their marks are decreased. Combining \eqref{eq:withtilde} with \eqref{eq:barfortilde}, the FKG-inequality and Lemma~\ref{lem:regular}, we obtain
\[
\P(A_{n-1}(1)\cap E(1)|\eta)\leq \bar{{\mathbf{P}}}_\eta(T \text{ is not $\mu$-regular})\leq v_{n-1}^{1-\mu}\e^{-v_{n-1}^\mu/8}.
\]
Integration over the point configurations $\eta$ and inserting the result into \eqref{eq:simplebnbound} yields
\[
b_n\leq \sigma_n^d v_{n-1}^{1-\mu}\e^{-v_{n-1}^\mu/8}
,\]
which concludes the proof.
\end{proof}
\begin{lemma}\label{lem:c_nbound}
Let 
\[
c_n= \P \left(\bigcup_{i,j=1}^{\sigma_n^d} B(i)\cap B(j) \cap F(i,j)\right),\quad n\in\N,
\]
then
\begin{align*}
	c_n
	\leq & \, 2\sigma_n^{2d} v_{n-1}^{1-\mu}\e^{-v_{n-1}^\mu/8} \\
	& \, + \sigma_n^{2d} \exp\Big(- C v_{n-1}^2 \int_{v_{n-1}^{-(1-\mu)}}^{1-v_{n-1}^{-(1-\mu)}}\int_{v_{n-1}^{-(1-\mu)}}^{1-v_{n-1}^{-(1-\mu)}}\varphi( s,t, \sqrt{d}m_n)\textup{d}s\textup{d}t\Big),
\end{align*}	
\end{lemma}
\begin{proof}
We use a similar approach as in the proof of Lemma~\ref{lem:b_nbound}, albeit we need to take a little more care, since the events involved are more complicated. Let $n\in\N$, $1\leq i<j\leq \sigma_n^d$, and the corresponding subcubes $\Gamma_n(i),\Gamma_n(j)\subset \Gamma_n$ be fixed. Our goal is to bound the probability $\P(B(i)\cap B(j) \cap F(i,j))$. As in the previous proof, we define additional randomly tagged locations that depend on edge and vertex marks. More precisely, set $Y(i)=Y(j)=\emptyset$ on $(B(i)\cap B(j))^{\mathsf{c}}$ and on $B(i)\cap B(j)$, we sample two sets of locations $Y(i)=\{Y_1(i),\dots,Y_{v_{n-1}}(i)\}$ from $\mathcal{R}^{n-1,k}(x_n(i))$ and $Y(j)=\{Y_1(j),\dots,Y_{v_{n-1}}(j)\}$ from $\mathcal{R}^{n-1,k}(x_n(j))$, respectively, uniformly and without replacement. The joint distribution of $\xi$ and the tagged sets $Y(i),Y(j)$ is denoted by $\tilde{\P}$. The vertex mark collections corresponding to $Y(i)$ and $Y(j)$ are denoted by $T(i)$ and $T(j)$, the corresponding vertex sets by $\mathbf{Y}(i),\mathbf{Y}(j)\subset \eta'$,  and the vertex mark sets associated with the independently tagged locations $X(x_n(i),n-1)$ and $X(x_n(j),n-1)$ are denoted by $S(i)$ and $S(j)$, respectively. The edge marks on potential edges between $X(x_n(i),n-1)$ and $X(x_n(j),n-1)$ are denoted by \[
V(i,j)=\{V_{s,t}(i,j), 1\leq s,t\leq v_{n-1}\}.
\]
 Arguing precisely as in the proof of Lemma~\ref{lem:b_nbound}, we find that
\begin{equation}\label{eq:tildetobar2}
\begin{aligned}
&\tilde{\mathbf{P}}_\eta((\xi,Y(i),Y(j))\in\cdot\;| B(i)\cap B(j) )\\
& =\tilde{\mathbf{P}}_\eta((\xi,X(x_n(i),n-1),X(x_n(j),n-1))\in\cdot\;| B(i)\cap B(j)\cap G),
\end{aligned}	
\end{equation}
where 
\begin{align*}
G &=\{\X(x_n(i),n-1)\subset \mathcal{R}^{n-1,k}(x_n(i))\}\cap \{\X(x_n(j),n-1)\subset \mathcal{R}^{n-1,k}(x_n(j))\} \\
& \subset\{S(i),S(j)\text{ are }\mu\text{-regular}\}
\end{align*}
and $\mathbf{P}_\eta,\tilde{\mathbf{P}}_\eta$ and $\bar{\mathbf{P}}_\eta$ denote the conditional versions of $\P,\tilde{\P}$ and $\bar{\P}$, respectively, given a fixed point configuration $\eta$.
We may thus rewrite
\begin{equation}\label{eq:qijrewrite}
\begin{aligned}
&{\mathbf{P}}_\eta(F(i,j)|B(i)\cap B(j))\\
&=\tilde{\mathbf{P}}_\eta(F(i,j)|B(i)\cap B(j))\\
&\leq \tilde{\mathbf{P}}_\eta(\mathbf{Y}(i)\not\leftrightarrow \mathbf{Y}(j)|B(i)\cap B(j))\\
&=\bar{\mathbf{P}}_\eta(\X(x_n(i),n-1)\not\leftrightarrow \X(x_n(j),n-1)|B(i)\cap B(j)\cap G).
\end{aligned}
\end{equation}
Let us say, that two vertices $\x=(x,s),\y=(y,t)$ with $x,y\in\eta\cap \Gamma_n$ are \emph{strongly connected}, if 
\[
V_{\x\y}\leq 1-\textup{e}^{\varphi(s,t,\sqrt{d}m_n)}.
\]
If $V,W\subset \eta'\cap \Gamma_n\times(0,1)$ are disjoint vertex sets, then we set
\[
\{V \leftrightharpoons W\} :=\{ \exists \x\in V, \y\in W:\; \x\text{ and }\y\text{ are strongly connected} \}.
\]
We further have, that
\begin{equation}\label{eq:strongconnectionbound}
\begin{aligned}
		\bar{\mathbf{P}}_\eta & (\X(x_n(i),n-1)\leftrightarrow \X(x_n(j),n-1)|B(i)\cap B(j)\cap G)\\
		\geq&\,\bar{\mathbf{P}}_\eta(\X(x_n(i),n-1)\leftrightharpoons \X(x_n(j),n-1)|B(i)\cap B(j)\cap G)\\
		=&\,\bar{\mathbf{P}}_\eta(\{\X(x_n(i),n-1)\leftrightharpoons \X(x_n(j),n-1)\}\\
		&\, \cap \{S(i),S(j)\text{ are }\mu\text{-regular}\}|B(i)\cap B(j)\cap G),
\end{aligned}
\end{equation}
since $\sqrt{d}m_n$ is an upper bound for the distance of any two vertices in $\Gamma_n$ and the event $\{S(i),S(j)\text{ are }\mu\text{-regular}\}$ almost surely occurs conditionally on $G$. 
In fact, denoting by $S^\neg=S^\neg(X)$ the collection of vertex marks not belonging to $S(i)\cup S(j)$ and by $V^\neg=V^\neg(X)$ the collection of edge marks belonging to edges with at least one endpoint mark not in $S(i)\cup S(j)$, we may disintegrate the measure $\bar{\mathbf{P}}_\eta$ as follows
$$
\bar{\mathbf{P}}_\eta = \int \bar{\mathbf{P}}_{\eta,(X,S^\neg,V^\neg)} \di (X, S^\neg,V^\neg).
$$
Almost every one of the measures $\bar{\mathbf{P}}_{\eta,(X,S^\neg,V^\neg)}$ coincides with
\begin{equation}\label{eq:defUniDist}
\operatorname{Uniform}(0,1)^{\otimes v_{n-1}^2+2v_{n-1}}
\end{equation}
and acts on the $\sigma$-field $\sigma(S(i),S(j),V(i,j))$.
Fix any configuration of coordinates $$(X,S^\neg,V^\neg)=\omega.$$ Let $\mathcal{X}_\eta$ denote the state space of tagged sub-configurations of $\eta$. For any function $$h_\eta:\mathcal{X}_\eta\times(0,1)^{\eta^{[2]}\times \eta}\to[0,\infty)$$ of tags, vertex and edge marks, we denote by $$h_\eta^\omega(\cdot)=h_\eta(\cdot,\omega)$$ its restriction to the coordinates \emph{not} prescribed by $\omega$, i.e.\ its restriction to the coordinates corresponding to $S(i)\cup S(j)\cup V(i,j)$. We argue now that the following $\sigma(S(i),S(j),V(i,j))$-measurable functions are increasing
\begin{enumerate}[(i)]
\item $\1\{\X(x_n(i),n-1)\leftrightharpoons \X(x_n(j),n-1) \text{ and } S(i),S(j)\text{ are }\mu\text{-regular}\}$,
\item $ \1^\omega_{B(i)\cap B(j)}$, and
\item $ \1_G^\omega.$
\end{enumerate}
Firstly, the event in (i) is itself measurable with respect to $\sigma(S(i),S(j),V(i,j))$ and increasing. Secondly, $B(i)\cap B(j)$ is clearly an increasing event with respect to the full configuration $\xi$ and thus increasing in the coordinates defined by $(S(i),S(j)$ and $V(i,j))$ for almost every realisation of tags $X$. Finally, for (iii) we note that the event $G$ is increasing in $S(i),S(j)$ and $V(i,j)$ for a fixed complementary configuration $\omega$, since if a tagged full configuration $\bar{\xi}$ satisfies $G$, then any configuration obtained by decreasing the mark of a tagged vertex in $S(i)\cup S(j)$ or a corresponding edge mark in $V(i,j)$ must also be in $G$. For vertex marks, this is checked as in the proof of Lemma~\ref{lem:b_nbound} under the additional provision that $(\mu,v_n)$-regularity be not violated for either maximal precluster, which follows from the monotonicity of that property in $S(i)$ and $S(j)$, respectively. For the edge marks $V(i,j)$, this follows from the fact that the graph $\G$ (and therefore the composition of the maximal preclusters) can only be affected if the edge corresponding to the mark is added. But since $\bar{\xi}$ satisfies $G$, this means adding an edge incident to both maximal $(\mu,v_{n-1})$-regular preclusters, which can only make those clusters larger and preserve maximality.  We remark that the event $G$ is \emph{not} monotone with respect to the full configuration $\xi$ but only with respect to the sub-$\sigma$-field $\sigma(S(i),S(j),V(i,j))$.

We conclude that we may apply the FKG-inequality for the product measure \eqref{eq:defUniDist} to obtain
\begin{equation}\label{eq:strongconnectionbound2a}
	\begin{aligned}
		\bar{\mathbf{P}}_\eta&(\X(x_n(i),n-1)\leftrightharpoons \X(x_n(j),n-1), S(i),S(j)\text{ are }\mu\text{-regular}| B(i)\cap B(j)\cap G)\\
		= & \int \bar{\mathbf{P}}_{\eta,(X,S^\neg,V^\neg)} \big(\{\X(x_n(i),n-1)\leftrightharpoons \X(x_n(j),n-1), S(i),S(j)\text{ are }\mu\text{-regular}\} \\
			& \cap G \cap B(i)\cap B(j) \big) \di (X, S^\neg,V^\neg) ({\bar{\mathbf{P}}_\eta(B(i)\cap B(j)\cap G)})^{-1}\\
		\geq & \int \bar{\mathbf{P}}_{\eta,(X,S^\neg,V^\neg)} \big(\X(x_n(i),n-1)\leftrightharpoons \X(x_n(j),n-1), S(i),S(j)\text{ are }\mu\text{-regular}\big) \\
		& \bar{\mathbf{P}}_{\eta,(X,S^\neg,V^\neg)} ( G \cap B(i)\cap B(j) ) \di (X, S^\neg,V^\neg) ({\bar{\mathbf{P}}_\eta(B(i)\cap B(j)\cap G)})^{-1}.
	\end{aligned}
\end{equation}
Lemma~\ref{lem:regular} and Lemma~\ref{lem:preciseconnection} together with the independence of edge and vertex marks imply that
\begin{equation}\label{eq:strongconnectionbound2}
	\begin{aligned}
		&\bar{\mathbf{P}}_{\eta,(X,S^\neg,V^\neg)} \big(\X(x_n(i),n-1)\leftrightharpoons \X(x_n(j),n-1) \text{ and } S(i),S(j)\text{ are }\mu\text{-regular}\big)\\
		& \geq \Big(1-2 v_{n-1}^{1-\mu}\e^{-v_{n-1}^\mu/8}\Big)\\
		& \phantom{space} \Big(1-\exp\Big[- C v_{n-1}^2 \int_{v_{n-1}^{-(1-\mu)}}^{1-v_{n-1}^{-(1-\mu)}}\int_{v_{n-1}^{-(1-\mu)}}^{1-v_{n-1}^{-(1-\mu)}}\varphi( s,t, \sqrt{d}m_n) \textup{d}s\textup{d}t\Big]\Big)\\
		&\geq 1-2 v_{n-1}^{1-\mu}\e^{-v_{n-1}^\mu/8}-\exp\Big[- C v_{n-1}^2 \int_{v_{n-1}^{-(1-\mu)}}^{1-v_{n-1}^{-(1-\mu)}}\int_{v_{n-1}^{-(1-\mu)}}^{1-v_{n-1}^{-(1-\mu)}}\varphi( s,t, \sqrt{d}m_n) \textup{d}s\textup{d}t\Big]
	\end{aligned}
\end{equation}
uniformly over all realisations of $\eta$ (and any complementary realisation $(X,S^\neg,V^\neg)=\omega$) that put sufficiently many points into the cubes $\Gamma_n(i),\Gamma_n(j)$. Substituting \eqref{eq:strongconnectionbound2} into \eqref{eq:strongconnectionbound2a}, eliminating the integral, and then applying the bound \eqref{eq:strongconnectionbound} yields the lower bound
\begin{equation}\label{eq:strongconnectionbound2b}
\begin{aligned}
		\bar{\mathbf{P}}_\eta&(\X(x_n(i),n-1)\leftrightarrow \X(x_n(j),n-1)| B(i)\cap B(j)\cap G)\\
		\geq & 1-2 v_{n-1}^{1-\mu}\e^{-v_{n-1}^\mu/8}-\exp\Big[- C v_{n-1}^2 \int_{v_{n-1}^{-(1-\mu)}}^{1-v_{n-1}^{-(1-\mu)}}\int_{v_{n-1}^{-(1-\mu)}}^{1-v_{n-1}^{-(1-\mu)}}\varphi( s,t, \sqrt{d}m_n) \textup{d}s\textup{d}t\Big].
\end{aligned}
\end{equation}
Combining the estimate \eqref{eq:strongconnectionbound2b} with \eqref{eq:qijrewrite} and integrating over point configurations $\eta$, we get
\begin{align*}
{\P}(F(i,j)|B(i)\cap B(j))\leq &\, 2 v_{n-1}^{1-\mu}\e^{-v_{n-1}^\mu/8}\\ 
&\, +\exp\Big(- C v_{n-1}^2 \int_{v_{n-1}^{-(1-\mu)}}^{1-v_{n-1}^{-(1-\mu)}}\int_{v_{n-1}^{-(1-\mu)}}^{1-v_{n-1}^{-(1-\mu)}}\varphi( s,t, \sqrt{d}m_n)\textup{d}s\textup{d}t\Big),
\end{align*}
and this estimate is uniform in the choice of subcubes $\Gamma_n(i),\Gamma_n(j)$. It follows that
\begin{align*}
c_n & \leq \sum_{i,j=1}^{\sigma_n^d} \P(B(i)\cap B(j) \cap F(i,j)) \leq \sum_{i,j=1}^{\sigma_n^d} \P(F(i,j)|B(i)\cap B(j) )\\
&\leq 2\sigma_n^{2d} v_{n-1}^{1-\mu}\e^{-v_{n-1}^\mu/8} + \sigma_n^{2d} \textup{e}^{- C v_{n-1}^2 \int_{v_{n-1}^{-(1-\mu)}}^{1-v_{n-1}^{-(1-\mu)}}\int_{v_{n-1}^{-(1-\mu)}}^{1-v_{n-1}^{-(1-\mu)}}\varphi( s,t, \sqrt{d}m_n)\textup{d}s\textup{d}t},
\end{align*}	
and the proof is concluded.
\end{proof}
We are now ready to prove Lemma~\ref{lem:alivebound}.
\begin{proof}[{Proof of Lemma~\ref{lem:alivebound}}]
By Lemma~\ref{lem:decomp}, we have
\[
a_n= \P(A_n^{\mathsf{c}})  = \P((A'_n)^{\mathsf{c}})+\P( (B_n\cap C_n)^{\mathsf{c}}\cap A'_n ),\quad n\in\N,
\]
thus combining Lemmas~\ref{lem:a_nbound}, \ref{lem:b_nbound} and \ref{lem:c_nbound} yields, for any $n\in\N$,
\begin{equation}\label{eq:auxbound0}
\begin{aligned}	
a_n\leq \frac{a_{n-1}}{1-\varrho_n} & + (2\sigma_n^{2d}+\sigma_n^d) v_{n-1}^{1-\mu}\e^{-v_{n-1}^\mu/8} \\
 & + \sigma_n^{2d} \exp\Big(- C v_{n-1}^2 \int_{v_{n-1}^{-(1-\mu)}}^{1-v_{n-1}^{-(1-\mu)}}\int_{v_{n-1}^{-(1-\mu)}}^{1-v_{n-1}^{-(1-\mu)}}\varphi( s,t, \sqrt{d}m_n)\textup{d}s\textup{d}t\Big)
\end{aligned}
\end{equation}
To bound the right hand side further, we first observe that, by definition of $\sigma_n$, 
\begin{equation}\label{eq:auxbound1}
2\sigma_n^{2d}+\sigma_n^d\leq 4 n^{2d\omega},
\end{equation}
for all but finitely many $n\in\N$. Since $v_n\to\infty$ as $n\to\infty$, we also have
\begin{equation}\label{eq:auxbound2}
v_{n-1}^{1-\mu}\e^{-v_{n-1}^\mu/8}\leq \e^{-v_{n-1}^\mu/9},
\end{equation}
for all sufficiently large $n$, and finally, we claim that
\begin{equation}\label{eq:auxbound3}
	\sqrt{d}m_n\leq v_{n-1}^{\nu/d},
\end{equation}
for all sufficiently many $n$, which we show at the end of the proof. Inserting \eqref{eq:auxbound1}--\eqref{eq:auxbound3} into \eqref{eq:auxbound0}, we obtain
\begin{align*}
a_n & \leq \frac{a_{n-1}}{1-\varrho_n} + 4n^{2d\omega}\Big(\e^{-v_{n-1}^\mu/9}+\exp\Big[- C v_{n-1}^2 \int_{v_{n-1}^{-(1-\mu)}}^{1-v_{n-1}^{-(1-\mu)}}\int_{v_{n-1}^{-(1-\mu)}}^{1-v_{n-1}^{-(1-\mu)}}\varphi( s,t, v_{n-1}^{\nu/d})\textup{d}s\textup{d}t \Big] \Big).
\end{align*}
Setting $\tilde{v}_{n-1}=v_{n-1}^{\nu}$ and using the definition of $\bar{\delta}_{\mathsf{eff}}(\mu^*)<2$ as well as $\mu<\mu^*$, we see that there exists some small $\zeta_0>0$ with $\bar{\delta}_{\mathsf{eff}}(\mu^*)+\zeta_0<2$, such that for every $\zeta\in(0,\zeta_0)$ there exists some $N(\zeta)$ such that for all $n>N(\zeta)$
\begin{equation*}
\begin{aligned}
a_n  & \leq \frac{a_{n-1}}{1-\varrho_n} + 4n^{2d\omega}\Big(\e^{-v_{n-1}^\mu/9}+\textup{e}^{- C \tilde{v}_{n-1}^{\frac{2}{\nu} -\bar{\delta}_{\mathsf{eff}}(\mu^*)- \zeta}} \Big) \\
& = \frac{a_{n-1}}{1-\varrho_n} + 4n^{2d\omega}\Big(\e^{-v_{n-1}^\mu/9}+\textup{e}^{- C {v}_{n-1}^{{2} -\nu\bar{\delta}_{\mathsf{eff}}(\mu^*)- \nu\zeta}}\Big).
\end{aligned}
\end{equation*}
Using that $\bar{\delta}_{\mathsf{eff}}(\mu^*)<2$ and the choice \eqref{eq:nudef} of $\nu$, it is easy to see that if we chose $\zeta$ small enough, we can find some small value $\mu_0$ (depending on $\mu,\nu$), such that for all sufficiently large $n$
\begin{equation*}
	\begin{aligned}
		a_n  \leq \frac{a_{n-1}}{1-\varrho_n} + 5n^{2d\omega}\e^{-v_{n-1}^{\mu_0}}.
	\end{aligned}
\end{equation*}
From the choice of $\omega$ in \eqref{eq:omegadef1} and the definition of $v_{n-1}$, it follows that $v_n$ grows at least like a small power of $n!$. Since $\rho_n$ decays only polynomially, we can find some large $L\in\N$, such that 
\[
a_n  \leq \frac{a_{n-1}}{1-\varrho_n} + a_0\varrho_n,\quad \text{for all }n>L,
\]
and since $\varrho_n<1/4$ for all $n\in\N$, we conclude that
\[
a_0\varrho_n + a_{n-1}(1+2\varrho_n)\leq (1+3\varrho_n)(a_0\vee a_{n-1}),\quad \text{for all }n>L,
\]
which by induction yields
$$a_n\leq (a_0 \vee a_L) \prod_{k=L+1}^n(1+3\varrho_k) \quad \text{for all }n>L.$$ Since $(\varrho_n)$ is summable, the product on the right hand side converges and we obtain the uniform bound \eqref{eq:mainprobbound} asserted in the lemma.

We conclude by verifying \eqref{eq:auxbound3}. Note that
\[
(\sqrt{d}m_n)^d= v_{n-1} \,\frac{2 d^{d/2} \sigma_n^d}{\theta \prod_{i=1}^{n-1} \varrho_i},
\]
and, writing $\nu=1+\varepsilon_\nu$ with $\varepsilon_\nu>0$, it is sufficient to show for all sufficiently large $n\in\N$,	
\[
\frac{ 2 d^{d/2} \sigma_n^d}{\theta \prod_{i=1}^{n-1} \varrho_i}\leq v_{n-1}^{\varepsilon_\nu},.
\]
which follows from the following calculation based on the choices of $\varrho_n$ and $\sigma_n$: We can find numbers $K,R\in\N$, such that for all $n>K+1$,
\begin{align*}
	L(n):=\log(v_{n-1}^{\varepsilon_\nu}) & = \log\left[(\theta m_0^d/2)^{\varepsilon_\nu} \prod_{i=1}^{n-1}\varrho_i^{\varepsilon_\nu}\sigma_i^{\varepsilon_\nu d}\right] \\
	& = \log\left[(\theta m_0^d/2)^{\varepsilon_\nu}\right]+ \varepsilon_\nu \sum_{i=1}^{n-1}(\log(\varrho_i)+ d\log(\sigma_i))\\
	&  \geq \log\left[2 d^{d/2}/\theta\right] + \varepsilon_\nu (d\omega -2) \sum_{i=K}^{n-1}\log i - R,
\end{align*}
Using the bound \eqref{eq:omegadef1}, we see that $\varepsilon_\nu (d\omega -2)>2$ and hence we can find $\varepsilon'$ such that
\[
L(n)\geq \log\left[2 d^{d/2}/\theta\right] + (2+\varepsilon') \sum_{i=1}^{n-1}\log i \geq \log\left[d^{d/2}/\theta\right] + d\log \sigma_n - \sum_{i=1}^{n-1}\log(\varrho_n),
\]
for all sufficiently large $n$, which concludes the proof.
\end{proof}
The remainder of this section is devoted to the completion of the proof of Theorem~\ref{thm:sublinearclusters}. 
\begin{proof}[Proof of Theorem~\ref{thm:sublinearclusters}]
	Let $\varepsilon\in(0,1/2)$ and $0<\lambda<1$ be given. Fix $\omega$ such that
	\begin{equation}\label{eq:omegadef2}
		\omega>\frac{2 \nu}{d(\nu-1)} \vee \frac{2}{d(1-\lambda)},
	\end{equation}
	and note that this condition implies \eqref{eq:omegadef1} and let the other parameters $\mu,\nu,(\sigma_n)$ and $(\rho_n)$ be defined as before. We have not yet specified the initial cube length $\ell=m_0$ and the parameter $k$ used in the definition of preclusters, which we do now. By ergodicity, we may choose $\ell$ so large, that with probability exceeding $1-\varepsilon/2$, there is a set $A$ of at least $\theta \ell^d /2$ vertices inside $\Gamma_0$ that belong to the infinite cluster. Since the infinite cluster is unique, there is some $k^\ast(\ell)<\infty$ such that all the vertices in $A$ are contained within the same cluster of $\G[\Delta_k\Gamma_0]$ with probability exceeding $1-\varepsilon/2$ if $k>k^\ast(\ell)$. We thus have shown that, with probability exceeding $1-\varepsilon$ we can find a $(v_0,\mu)$-precluster in $\Gamma_0$ and conclude that $a_0=\P(\Gamma_0\text{ is not alive})\leq \varepsilon.$ Invoking Lemma~\ref{lem:alivebound}, we now obtain that 
	\[
	\P(\Gamma^n \text{ is not alive})\leq C \varepsilon, \quad n\in\N,
	\]
	where $C$ depends only on $\varphi$ and $\ell$.	If the stage-$L$ cube $\Gamma_L$ is alive, then it follows from the definitions of aliveness and preclusters, that $\Delta_k\Gamma_L$ contains a cluster of size
	\[
	v_{L}= \frac{\theta \ell^d}{2}\prod_{i=1}^{L} {\varrho_i}\sigma_i^d.
	\]
	Let $\tilde{\Gamma}_L$ denote the union of $\Gamma_L$ with the $3^d-1$ stage-$L$ cubes neighbouring it. Since $k$ depends only on the initial cube size $\ell$, we can find $L_0\in\N$ such that
\[
\Delta_k\Gamma_L\subset\tilde\Gamma_L,\quad \text{for all }L\geq L_0.
\] From the choice of $(\varrho_n),(\sigma_n)$ and \eqref{eq:omegadef2} we can also deduce the existence of constants $0<q_1,q_2,q_3,q_4<\infty$, such that
	\begin{align*}
		v_{L}\geq q_1\, \frac{\theta \ell^d}{2}\prod_{n=1}^L n^{d\omega -2}\geq q_2\, \frac{\theta m^d}{2}\left(\prod_{n=1}^L \sigma_n^d\right)^{\frac{d\omega-2}{d\omega}}=q_3\, \theta \textup{vol}({\Gamma}_L)^{1-\frac{2}{d\omega}}>q_4\textup{vol}(\tilde{\Gamma}_L)^{\lambda},
	\end{align*}
	for $L$ sufficiently large. Since we can choose $\varepsilon$ arbitrarily close to $0$ and $\lambda$ arbitrarily close to $1$, the conclusion of Theorem~\ref{thm:sublinearclusters} now follows easily for the subsequence $(\P(|\C_{\mathsf{max}}(\G[\Gamma^{m_n}])|> m_n^{\lambda d}), n\in\N)$. To obtain the result for the original sequence, fix $\varepsilon>0$ and $1>\lambda'>\lambda$ arbitrarily. Now choose $N$ so large that for all $k\geq N$ both $$\P(|\C_{\mathsf{max}}(\G[\Gamma^{m_k}])|> m_k^{\lambda' d})>1-\varepsilon \quad \text{ and } \quad m_k^{\lambda'd}>m_{k+1}^{\lambda d}$$ are satisfied. Then, if $n\geq M_n$, we can always find $k$ with $m_k\leq n<m_{k+1}$ and such that $$\P(|\C_{\mathsf{max}}(\G[\Gamma^{n}])|> n^{\lambda d})\geq \P(|\C_{\mathsf{max}}(\G[\Gamma^{n}])|> m_n^{\lambda' d}) \geq \P(|\C_{\mathsf{max}}(\G[\Gamma^{m_n}])|> m_n^{\lambda' d})\geq 1-\varepsilon,$$
and the proof is complete.
\end{proof}	
\section{Transience}\label{sec:transience}
We show that $\C_\infty$ is transient by explicitly constructing a transient subgraph. To this end, we use the notion of \emph{renormalised graphs} \cite[Def.\ 2.8]{Berge02}. 
\begin{definition}\label{def:grs} 
	An \emph{$\ell$-merger} of an infinite graph $\mathcal{H}$ is any graph $\mathcal{H}'$ obtainable by partitioning $V(\mathcal{H})$ into subsets $v_1,v_2,\dots$ of size $\ell$, setting $V(\mathcal{H}')=\{v_i,i\in\N\}$ and $v_iv_j\in E(\mathcal{H}')$ if and only if there are $u_i\in v_i,u_j\in v_j$ with $u_iu_j\in E(\mathcal{H})$. The graph $G_0=(V_0,E_0)$ is \emph{renormalised for the sequence} $(\ell_n)_{n\in\N}$ if we can construct a sequence of $G_1,G_2,\dots$ of graphs with $G_i=(V_i,E_i)$ such that
	\begin{itemize}
		\item $G_i$ is an $\ell_i$-merger of $G_{i-1}$ for all $i\in\N$;
		\item for every $i\geq 2$, there is a partition of $V_i$ into subsets $v_1,v_2,\dots$ of size $\ell_{i+1}$ such that for every $k$ and every pair $(u_1,u_2)\in v_k\times v_k$ (interpreted as subsets of $V_{i-1}$) and every distinct $w_1\in u_1,w_2\in u_2$ (interpreted as a pair of subsets of $V_{i-2}$), we have that $w_1\leftrightarrow w_2$ in $G_0$ (interpreted as subsets of $V_0$).
	\end{itemize} 	
\end{definition}
\begin{remark}
The wording of our definition of renormalised graphs differs from \cite[Def.\ 2.8]{Berge02}, but it is straightforward to check that the two formulations are in fact equivalent.
\end{remark}

\begin{prop}\label{prop:RGS}
If $\eta$ is either a Poisson process or an independently percolated lattice, $\varphi$ is such that $\bar{\delta}_{\mathsf{eff}}(0+)<2$ and $\theta>0$, then $\G$ contains a transient subgraph almost surely.
\end{prop}
Let us prepare the proof of Proposition~\ref{prop:RGS} with some notation. Define for \(n\in\N\) the values
\begin{align*}
	\alpha_n&:=\left\lceil(n+1)^{2\lambda d}\right\rceil,\quad \sigma_n:=(n+1)^2,
\end{align*}
where $\lambda\in(1/2,1)$. Our goal is to show that $\C_\infty$ contains a subgraph that is renormalised for $(\alpha_n)_{n\in\N}$, by \cite[Lemma 2.7]{Berge02}, this implies transience.\medskip

Fix furthermore a parameter $\mu\in(0,1/2)$ which governs the regularity of vertex weights just as in the previous section. Once again, the scaling sequence $(\sigma_n)$ tells us how fast the scales of the renormalisation scheme grow, but the construction of connections between clusters at different scales will be significantly different to make it compatible with Definition~\ref{def:grs}. Recall that $\mathfrak{C}(m)$ denotes the collection of disjoint cubes of side-length $m$ centred at the points of $m\Z^d$. A cube
\[
\Gamma\in \mathfrak{C}^n := \mathfrak{C}\left(\prod_{i=1}^n \sigma_i\right),
\]
for $n\geq 1$, is called a \emph{stage-$n$ cube}. We now define the procedure that will allow us to conclude the renormalisability of $\C_\infty$ for $(\alpha_n)$. The scheme requires us to start at some sufficiently large scale, so let \(n_1\in\N\) be the smallest stage of cubes we will consider. We do not fix \(n_1\) yet and assume only that it is large.
We now define what it means for a cube to be \emph{good} or \emph{bad}. We begin at the bottom levels, namely $\Gamma\in \mathfrak{C}^{n_1}$ is \emph{good}  if
\begin{description}
	\item[$\mathsf{E}(n_1)$] The subgraph $\G[\Gamma]$ contains a $\left(\prod_{i=1}^{n_1}\alpha_i, \mu\right)$-regular connected component.
\end{description}
We call a connected component of $\G[\Gamma]$ which realises condition $\mathsf{E}(n_1)$ an \emph{$n_1$-renormalised cluster}. 
Moving to level $n_1+1$, we declare an $(n_1+1)$-level cube $\Gamma$ \emph{good}, if
\begin{description}
	\item[$\mathsf{E}_1(n_1+1)$] it has at least $\alpha_{n_1+1}$ good subcubes at level $n_1$;
	\item[$\mathsf{E}_2(n_1+1)$] at least $\alpha_{n_1+1}$ good subcubes of $\Gamma$ at level $n_1$  contain $n_1$-renormalised clusters which are all mutually adjacent in $\G[\Gamma];$
	\item[$\mathsf{E}_3(n_1+1)$] the adjacency requirement $\mathsf{E}_2(n_1+1)$ produces at least one cluster in $\mathcal{G}[\Gamma]$ that is $(\prod_{i=1}^{n_1+1}\alpha_i, \mu)$-regular.
\end{description}
A cluster qualifying for the condition imposed in $\mathsf{E}_3(n_1+1)$ is called an \emph{$(n_1+1)$-{renormalised}} cluster.\medskip	 

Having declared what happens at the bottom levels, we are now prepared to initiate the recursive part of the scheme. For a given level-$n$ cube $\Gamma$ with $n>n_1+1$, we say that a pair $(\Gamma_1,\Gamma_2)$ of good subcubes is $\emph{well-connected}$, if 
\begin{itemize}
	\item there exist $(n-2)$-{renormalised} 
	clusters $\mathcal{C}^{i}_1,\dots,\mathcal{C}^{i}_{\alpha_{n-2}}$ inside $\Gamma_i$, for $i=1,2$;
	
	\item each of the pairs of sets $\{ \mathcal{C}^{1}_k,\mathcal{C}^{2}_l: 1\leq k,l\leq \alpha_{n-2} \}$ is adjacent in $\mathcal{G}[\Gamma]$.
\end{itemize}
Based on the goodness at levels $n_1,n_1+1$ and the notion of well-connectedness, we now declare an stage-\(n\) cube \(\Gamma\) with $n>n_1+1$ to be \emph{good}, 
if there exists a subset $\mathfrak{F}$ of the good subcubes of $\Gamma$ satisfying
\begin{description}
	\item[$\mathsf{F}_1(n)$] $|\mathfrak{F}|\geq \alpha_n$;
	\item[$\mathsf{F}_2(n)$] any pair $(\Gamma_1,\Gamma_2)\in \mathfrak{F}\times\mathfrak{F}$ is well-connected;
	\item[$\mathsf{F}_3(n)$] at least one cluster formed by mutual well-connectedness of $\{(\Gamma_1,\Gamma_2):\Gamma_{1},\Gamma_{2}\in\mathfrak{F}\}$ 
	from the $(n-2)$-{renormalised} clusters inside subcubes of $\Gamma_1,\Gamma_2$ is $(\prod_{i=1}^{n}\alpha_i,\mu)$-regular.
\end{description}%
Any cluster formed from $(n-2)$-{renormalised}
clusters in the way specified by $\mathsf{F}_1(n)$--$\mathsf{F}_3(n)$ is called an \emph{$n$-renormalised cluster} of $\Gamma$.
Observe that the event $\{\Gamma\in\mathfrak{C}^n\}$ is good is increasing. 
\begin{remark}
The recursive architecture induced by a hierarchy of good cubes is more complex than the one used in the proof of Theorem~\ref{thm:sublinearclusters}, due to the intertwining of levels $n$ and $n+2$. However, note that the goodness of $\Gamma$ only depends on $\G[\Gamma]$ and therefore is independent of the status of cubes on the same level, if $\eta$ is the (percolated) lattice or a Poisson process.
\end{remark}
A careful inspection of the above construction yields that it produces indeed a renormalised graph sequence.
\begin{lemma}\label{lem:towerisgood}
	If $\Gamma_n, n\geq n_1$ are all good, then $\C_\infty$ contains a subgraph that is renormalised for $(\alpha_n)$.
\end{lemma}

Having defined our renormalisation scheme subject to the precise choice of the parameters $\lambda,\mu$ and $n_1$, we are now ready to complete the proof.
\begin{proof}[Proof of {Lemma~\ref{prop:RGS}}]
	The calculation is very similar to the proof of Theorem~\ref{thm:sublinearclusters}, which allows us to recycle some of the parameters chosen there. Let, in particular, $\mu^\ast>0$ be given as in \eqref{eq:mustardef} and define $\nu>1$ and $\mu<\mu^\ast$ as in \eqref{eq:nudef} and \eqref{eq:mudef}. Finally, choose
	\begin{equation}\label{eq:lambdadef}
		\lambda\in\left(\frac{1}{2 \wedge \mu},1\right).
	\end{equation}

	Let $E_n,E^i_n,F^i_n, i\in\{1,2,3\},$ denote the events that the conditions $\mathsf{E}(n),\mathsf{E}_i(n),\mathsf{F}_i(n)$ are satisfied respectively for the stage-$n$ cube $\Gamma^n$ centred at the origin. Similarly, we write $L_n, n\geq n_1$ for the event that $\Gamma^n$ is good. Note that transience of $\C_\infty$ has either probability $0$ or $1$ by ergodicity. Due to translation invariance and Lemma~\ref{lem:towerisgood} it is thus enough to show that
	\begin{align*}
		\P\Big(\bigcap_{n=n_1}^\infty L_n\Big)>0,
	\end{align*}
	and since the events \(L_n\) are increasing, this can be further simplified, using the FKG inequality, to 
	\begin{align*}
		\prod_{n=n_1}^\infty\P(L_n)>0.
	\end{align*}
	Note that by construction, \(L_{n_1}\) and \(L_{n_1+1}\) are defined differently than the other scales, so we will bound their probabilities separately. We first upper bound the probability of the converse event \(L_n^{\mathsf{c}}\) for \(n>n_1+1\) and begin by writing
	\begin{align}\label{eq:lnc}
		\P(L_n^{\mathsf{c}})= \P((F^1_n)^{\mathsf{c}})+ \P((F^2_n)^{\mathsf{c}}\cap F^1_n)+ \P((F^3_n)^{\mathsf{c}}\cap (F^1_n\cap F^2_n)).
	\end{align}
	To bound $\P((F^3_n)^{\mathsf{c}}\cap (F^1_n\cap F^2_n))$, one may argue as in the proof of Lemma~\ref{lem:b_nbound} to obtain
	\begin{equation*}
		\P((F^3_n)^{\mathsf{c}}\cap (F^1_n\cap F^2_n))\leq \exp\left(-\frac{1}{8}\prod_{i=1}^n \alpha_i^\mu\right)\prod_{i=1}^n \alpha_i^{1-\mu}.	
	\end{equation*}
	It follows that, for $n$ sufficiently large,
	\begin{equation}\label{eq:F3bound}
		\P((F^3_n)^{\mathsf{c}}\cap (F^1_n\cap F^2_n))\leq \e^{-C((n-1)!)^c}
	\end{equation}
	for some constants $c,C>0$ which only depend on $\mu.$ Moving on to bound $\P((F^2_n)^{\mathsf{c}}\cap F^1_n)$, note that any two vertices in $\Gamma^n$ are at most at distance
	\begin{align*}
		\sqrt{d}\prod_{k=1}^n \sigma_k=\sqrt{d}((n+1)!)^2
	\end{align*}
	away from each other. Let \(q_n\) be the probability that two renormalised clusters of \((n-2)\)-level subcubes in the same \(n\)-level cube are not connected. We note that by construction, any $(n-2)$-renormalised cluster in any good \((n-2)\)-level cube contains at least \(\prod_{i=1}^{n-2}\alpha_i\) vertices. Therefore, by \Cref{lem:preciseconnection}, and noting that $\lambda>1/\nu$ implies
	\[ \left(\prod_{i=1}^{n-2}\alpha_i \right)^\nu \geq ((n-1)!)^{2\lambda \nu d} > \sqrt{d}((n+1)!)^2 \text{ for all sufficiently large }n,\] 
	we may argue precisely as in the proof of Lemma~\ref{lem:c_nbound} to obtain that
	\begin{align*}
		q_n
		&\leq\exp\{-C((n-1)!)^c\}
	\end{align*}
	for some positive constants \(c,C\).
	There are 
	\begin{align*}
		{\sigma_n^d \sigma_{n-1}^d\choose 2 }<4^d(n+1)^{4d}
	\end{align*}
	possible pairs of \(n-2\)-level cubes in $\Gamma^n$. Taking a union bound, we obtain, for some $\tilde{c}>0$,
	\begin{equation}\label{eq:fnc}
\begin{aligned}		
	\P((F^2_n)^{\mathsf{c}}\cap F^1_n) & \leq \exp\{d\log(4)+4d\log(n+1)- c((n-1)!)^C\}\\
	 &\leq \exp\{{ -\tilde c((n-1)!)^C}\},
\end{aligned}
\end{equation}
	where the second inequality holds for all sufficiently large \(n\).
	We now proceed to bound \(\P(F^1_n)\). Note that by independence, the number of good subcubes of $\Gamma$ dominates a $\operatorname{Bin}(m,q)$ random variable $X$, where \(q=\mathbb{P}(L_{n-1})\), \(m=\sigma_n^d\). Fixing $\Theta\in(0,1)$, Chernoff's bound states \[\P(X<(1-\Theta)mq)\leq\exp\{-\frac{1}{2}\Theta^2mq\},\] i.e.\ for \(\Theta=1-\frac{\alpha_n}{\sigma_n^d}\frac{1}{\mathbb{P}(L_{n-1})}\) this leads to
	\begin{align}\label{eq:enc}
		\P(F^1_n)&\geq1-\exp\Big\{-\frac{1}{2}\Big(1-(n+1)^{2d(\lambda-1)}\frac{1}{\P(L_{n-1})}\Big)^2\P(L_{n-1})(n+1)^{2d}\Big\},\nonumber\\
		&=1-\exp\Big\{-\frac{1}{2}\big((n+1)^{2d(1-\lambda)}\P(L_{n-1})-1\big)^2(n+1)^{4d\lambda-2d}\P(L_{n-1})^{-1}\Big\}\nonumber\\
		&\geq 1-\exp\Big\{-\frac{1}{2}((n+1)^{2d(1-\lambda)}\P(L_{n-1})-1)^2(n+1)^{4d\lambda-2d}\Big\}
	\end{align}
	where we used the definitions of \(\alpha_n\), \(\sigma_n\) and the definition of \(F^1_n\) itself. Combining \eqref{eq:F3bound}, (\ref{eq:enc}) and (\ref{eq:fnc}) into (\ref{eq:lnc}) and relabelling constants, we obtain the recursive inequality
	\begin{align*}
		\P(L_n^{\mathsf{c}})&\leq2\exp\{{ - c((n-1)!)^C}\}
		+\exp\left\{-\frac{1}{2}\big((n+1)^{2d(1-\lambda)}\P(L_{n-1})-1\big)^2(n+1)^{4d\lambda-2d}\right\}.
	\end{align*}
	The same bound applies to \(\P(L_{n_1+1}^c)\), since the calculation for the complements of the defining events \(E^1_{n_1+1},E^2_{n_1+1}\) and $E^3_{n_1+1}$ can be done along the same lines as for the $F^i_n$. Finally, we have by Theorem~\ref{thm:sublinearclusters} that for any \(\lambda\in(0,1)\) and any \(\varepsilon>0\), \(\P(L_{n_1})>1-\varepsilon\) if \(n_1\) is chosen sufficiently large.\medskip
	
	Define now the sequence \(\ell_n:=1-(n+1)^{-3/2}\) and observe that \(\prod_{i=1}^{\infty}\ell_i>0\).
	We will now show that if \(\P(L_n)>\ell_n\), then it follows inductively that \(\P(L_{n+1})>\ell_{n+1}\). We calculate
	\begin{align*}
		\P(L_n^{\mathsf{c}})&\leq \exp\{{ - \tilde c((n-1)!)^C}\}
		+\exp\{-\frac{1}{2}\big((n+1)^{2d(1-\lambda)}(1-2^{-3/2})-1\big)^2(n+1)^{4d\lambda-2d}\}\\
		&\leq(n+1)^{-3/2}\\
		&=1-\ell_n,
	\end{align*}
	where the second inequality holds if \(n\) is sufficiently large. Let \(n_1\) now be large enough so that all \(n>n_1\) satisfy the previous assumptions about \(n\) being large and furthermore let \(n_1\) be large enough that \(\P(L_{n_1})>1-2^{-3/2}\). Then, using the same calculation yields that \(\P(L_{n_1+1})>\ell_{n_1+1}\) and the claim follows for all larger \(n\).
	
	We can now write 
	\begin{align*}
		\prod_{n=n_1}^\infty\P(L_n)=\P(L_{n_1})\prod_{n=n_1+1}^\infty\P(L_n)\geq\P(L_{n_1})\prod_{n=n_1+1}^\infty\ell_n>0.
	\end{align*}
	Together with \Cref{lem:towerisgood}, this gives the existence of the renormalized graph sequence with positive probability and concludes the proof.
\end{proof}

\section{Continuity properties of percolation}\label{sec:continuity}
For a given graph $\mathcal{G}$, we denote by $p\mathcal{G}$ the graph obtained from $\mathcal{G}$ by independent Bernoulli vertex percolation and that $\eta$ is finite range, if $\eta(A)$ and $\eta(B)$ are independent, whenever $A$ and $B$ are sufficiently far separated.
\begin{prop}(Continuity of percolation from the left, $d\geq 2$)\label{prop:leftcont}
Let $d\geq 2$ and let $\G$ be an instance of inhomogeneous long range percolation on a finite range point set $\eta$ with $\bar\delta_{\mathsf{eff}}(0+)<2$. Assume that an infinite cluster exists. Then there exists $p<1$ such that $p\G$ contains an infinite cluster almost surely.
\end{prop}
\begin{proof} We renormalise the model using the cubes $\mathfrak{C}(m)$. By the finite range assumption on $\eta$, we my fix $m_0$ so large, that $\eta(\Gamma^{m}(x))$ and $\eta(\Gamma^{m}(y))$ are independent for all $m\geq m_0$ and all $x,y\in m\Z^d$ with $|x-y|\geq 3m$. Using the classical domination result of Liggett et al.\ \cite{LiggeSchonStace97}, it is straightforward to show that there are retention probabilities $p^\ast_{3,d},q^\ast_{3,d}<1$, such that any ergodic $3$-independent $(p,q)$-site-bond percolation measure on $\Z^d$ with $p>p^\ast_{3,d}$ and $q>q\ast_{3,d}$ almost surely produces an infinite cluster.
	
A $\lambda$-\emph{good} cube $\Gamma\in \mathfrak{C}(m)$ is such that $\G[\Gamma]$ contains a $\mu(\lambda)$-regular cluster of size at least $m^{\lambda d}$, where $\mu(\lambda)$ is chosen such that $$\bar\delta_{\mathsf{eff}}\big(1-\lambda(1-\mu(\lambda))\big)=:\delta_\lambda<2.$$ Note that by assumption such a choice is always possible, if $\lambda$ is sufficiently close to $1$. By Theorem~\ref{thm:sublinearclusters} and the fact that $\mu$-regularity is a monotone event for a given cube, we can choose $m$ so large, that the density $p_\lambda$ of $\lambda$-good cubes is arbitrarily close to $1$. In particular, we can achieve $p_\lambda\in(p^\ast_{2,d},1)$. By Lemma~\ref{lem:preciseconnection} the maximal local clusters in two good cubes associated with neighbouring vertices in $m\Z^d$ are connected with probability at least
\[
q_\lambda:=1-\exp\left(-{Cm^{2d\lambda}\int_{1/m^{\lambda d(1-\mu(\lambda))}}^{1/2}}\int_{1/m^{\lambda d(1-\mu(\lambda))}}^{1/2} \phi(s,t,2\sqrt{d}m)\,\textup{d}s\textup{d}t\right),
\]
and we obtain, by choice of $\lambda$ and $\mu(\lambda)$
\[
q_\lambda\geq 1-\e^{-m^{\lambda d(2-\delta_{\lambda})(1+o(1))}}>q^\ast_{3,d},
\]	
for large enough $m$ and this estimate holds independently for disjoint pairs of $\lambda$-good cubes. Now given a large value $m$ such that the above estimates hold, we observe that 
\begin{itemize}
	\item the induced site-bond percolation model on cubes in $\mathfrak{C}(m)$ percolates, since it dominates a $3$-independent percolating bond-site model on $\Z^d$,
	\item if $p_\lambda (r)$ denotes the probability that a $\lambda$-good cluster exists in an $m$-cube for the percolated graph $r\G$, then we have 
	\[
	\lim_{r\uparrow 1} p_\lambda (r)=p_{\lambda},
	\]
	since the involved events only depend on a finite domain,
	\item the lower bound on $q_\lambda$ remains valid for any $\lambda$-good pair of neighbouring $m$-cubes in $r\G$.
\end{itemize}
Hence, choosing $r$	sufficiently close to $1$, we obtain that the induced site-bond percolation model on cubes in $\mathfrak{C}(m)$ still percolates for $r\G$ and thus there exists an infinite cluster in $r\G$.
\end{proof}
The following corollary establishes Theorem~\ref{thm:truncation}.
\begin{corollary}[Continuity under truncation]
Let $d\geq 2$ and $\bar{\delta}_{\mathsf{eff}}(0+)<2$. If $\G$ is has an infinite cluster, then there exists some $\ell<\infty$ such that $\G$ retains an infinite cluster after removing all edges of length at least $\ell$ from $\G$.
\end{corollary}
\begin{proof}
The auxiliary infinite cluster constructed from the coarse-grained bond-site model on $m$-cubes in the proof of Proposition~\ref{prop:leftcont} uses no edge longer than $2\sqrt{d}m$. 
\end{proof}
Finally, we establish continuity of the percolation function in all points.
\begin{proof}[Proof of Theorem~\ref{thm:continuity}]
Recall that $\G_p$ denotes the graph obtained from $\G$ by independent Bernoulli bond percolation with retention parameter $p\in [0,1]$. Since $\G$ is locally finite, it is elementary, to see that 
\[
\theta(p)=\P(\G_p \text{ contains an infinite cluster})
\]
is a right-continuous function of $p$, since $\theta(p)$ can be written as a decreasing limit of polynomials in $p$. Furthermore, a classic result of van den Berg and Keane \cite{vdBerKeane84} states that $\theta(\cdot)$ is continuous above the critical threshold. By monotonicity, it follows that $\theta(p)$ is continuous, if it is left-continuous at the critical value. This follows from Proposition~\ref{prop:leftcont} for site-percolation and an elementary coupling argument between explorations in bond percolation clusters and explorations in site percolation clusters shows that the conclusion of Proposition~\ref{prop:leftcont} remains true for bond percolation.
\end{proof}

\footnotesize{\printbibliography}

\end{document}